\documentclass[a4paper]{article}

\usepackage[utf8]{inputenc}
\usepackage[T1]{fontenc}
\usepackage{csquotes}
\usepackage[english]{babel}

\usepackage{bm}

\usepackage[backend=bibtex,style=alphabetic,giveninits=true,isbn=false,doi=false,url=false]{biblatex}
\usepackage{lmodern}

\usepackage{amssymb,amsmath,amsthm,mathrsfs}

\usepackage{todonotes}
\usepackage{mathtools,bbm,bm}
\usepackage[normalem]{ulem}
\usepackage{cancel}
\usepackage{verbatim}
\usepackage{tikz}
\usetikzlibrary{matrix}
\pgfdeclarelayer{background} 
\pgfsetlayers{background,main} 
\usepackage[shortlabels]{enumitem}

\usepackage{hyperref}

\usepackage{a4}
\usepackage{multirow}
\usepackage[footnotesize,bf,centerlast]{caption}
\usepackage{xcolor,colortbl}
\definecolor{Gray}{gray}{0.80}
\definecolor{LightGray}{gray}{0.90}
\definecolor{darkpastelgreen}{rgb}{0.01, 0.75, 0.24}

\usepackage{tikz}  
\usetikzlibrary{matrix} 
\pgfdeclarelayer{background} 
\pgfsetlayers{background,main} 

\newcommand{\cA}{\mathcal{A}}

\newcommand{\cC}{\mathcal{C}}
\newcommand{\cD}{\mathcal{D}}
\newcommand{\cE}{\mathcal{E}}

\newcommand{\cI}{\mathcal{I}}

\newcommand{\cL}{\mathcal{L}}

\newcommand{\cP}{\mathcal{P}}

\newcommand{\cT}{\mathcal{T}}
\newcommand{\cU}{\mathcal{U}}
\newcommand{\cV}{\mathcal{V}}
\newcommand{\cW}{\mathcal{W}}

\newcommand{\bN}{\mathbb{N}}

\newcommand{\bR}{\mathbb{R}}

\newcommand{\dd}{ \mathrm{d}}

\newcommand{\eps}{\varepsilon}

\newcommand{\vn}[1]{\left| \! \left| #1\right| \!\right|}

\newcommand{\ip}[2]{\langle #1,#2\rangle}

\newcommand{\otmap}[2]{\bm{t}^{{#2}}_{#1}}
\newcommand{\botmap}[2]{\tilde{\bm{t}}^{{#2}}_{#1}}

\numberwithin{equation}{section}

\newtheorem{theorem}{Theorem}[section]
\newtheorem{lemma}[theorem]{Lemma}
\newtheorem{proposition}[theorem]{Proposition}
\newtheorem{corollary}[theorem]{Corollary}

\theoremstyle{definition}
\newtheorem{definition}[theorem]{Definition}
\newtheorem{remark}[theorem]{Remark}

\newtheorem{assumption}[theorem]{Assumption}

\newcommand{\R}{\mathbb{R}}
\newcommand{\De}{\mathrm{d}}
\newcommand{\otgrad}{\mathrm{grad}^{W_2}}

\newcommand{\adm}{\mathrm{Adm}}
\newcommand{\admerg}{\mathrm{Adm}_{\infty}}
\newcommand{\admt}{\mathrm{Adm}_{T}}

\newcommand{\X}{\mathsf X}
\newcommand{\sfd}{\mathsf d}

\newcommand{\sfS}{\mathsf S}

\newcommand{\loc}{\mathrm{loc}}

\newcommand{\EVI}{\mathrm{EVI}}

\title{Hamilton--Jacobi equations for Wasserstein controlled gradient flows: existence of viscosity solutions}

\author{Conforti G.\thanks{CMAP, Ecole Polytechnique, Route de Saclay, 91128, Palaiseau Cedex, France. \emph{E-mail address}: giovanni.conforti@polytechnique.edu. Research supported by the ANR project  ANR-20-CE40-0014.} , Kraaij  R. C.\thanks{Delft Institute of Applied Mathematics, Delft University of Technology, Mekelweg 4, 2628 CD Delft, The Netherlands. \emph{E-mail address}: r.c.kraaij@tudelft.nl} , Tamanini L.\thanks{Universit\`a Cattolica del Sacro Cuore, Dipartimento di Matematica e Fisica
"Niccolo Tartaglia", Via della Garzetta 48, I-25133 Brescia, Italy. \emph{E-mail address}: luca.tamanini@unicatt.it. While this work was written LT was associated to INdAM and the group GNAMPA. Research supported by the project PRIN 2017 (prot.2017TEXA3H) "Gradient flows, Optimal Transport and Metric Structures"} , Tonon D.\thanks{Dipartimento di Matematica "Tullio Levi-Civita", Universit\`a degli Studi di Padova, via Trieste 63, 35121 Padova, Italy. \emph{E-mail address}: daniela.tonon@unipd.it. Research supported by the project King Abdullah University of Science and Technology (KAUST)  ORA-CRG2021-4674
“Mean-Field Games: models, theory and computational aspects”; by the project SID BIRD 2022 "Stochastic mean field control and the Schrödinger problem"; and by the project PRIN2022 (prot.2022W58BJ5)  "PDEs and optimal control methods in mean field games, population dynamics and multi-agent models"}}
\date{\today}

\begin{document}

\maketitle

\abstract{This work is the third part of a program initiated in \cite{CoKrTo21,CoKrTo23} aiming at the development of an intrinsic geometric well-posedness theory for Hamilton-Jacobi equations related to controlled gradient flow problems in metric spaces. In this paper, we finish our analysis in the context of Wasserstein gradient flows with underlying energy functional satisfying McCann's condition. More prescisely, we establish that the value function for a linearly controlled gradient flow problem whose running cost is quadratic in the control variable and just continuous in the state variable yields a viscosity solution to the Hamilton-Jacobi equation in terms of two operators introduced in our former works, acting as rigorous upper and lower bounds for the formal Hamiltonian at hand. The definition of these operators is directly inspired by the Evolutional Variational Inequality formulation of gradient flows (EVI): one of the main innovations of this work is to introduce a controlled version of EVI, which turns out to be crucial in establishing regularity properties, energy and metric bounds along optimzing sequences in the controlled gradient flow problem that defines the candidate solution.
}

\tableofcontents

{
\section{Introduction}
Given an energy functional $\cE$ on the Wasserstein space $\cP_2(\mathbb{R}^d)$ satisfying McCann's curvature condition \cite{mccann1997convexity}, we consider in this article the problem of controlling the corresponding gradient flow over an infinite time horizon in such a way that an optimal compromise is found between the cost of controlling and the rewards that applying a control allows to collect. Our main contribution is to show that the value function of the control problem provides with viscosity solutions for the corresponding Hamilton-Jacobi equation according to the notion of solution we introduced in \cite{CoKrTo23}. Since uniqueness  of viscosity solutions has been established in our previous works \cite{CoKrTo21,CoKrTo23}, the existence result of this work constitutes the third and final step in completing the program of showing well posedness for this class of equations. In a nutshell, given an energy functional on the Wasserstein space $\cE:\cP_2(\mathbb{R}^d)\rightarrow\mathbb{R}^d$, whose (formal) Wasserstein gradient we denote by $\otgrad$, the equation of interest is 
\begin{equation}\label{eq:HJ_intro}
f + \langle \otgrad \cE , \otgrad f \rangle_{\mathrm{T}_{\mu}\mathcal{P}_2(\mathbb{R}^d)}  - \frac12\|\otgrad f \|^2_{\mathrm{T}_{\mu}\mathcal{P}_2(\mathbb{R}^d)}=h,
\end{equation}
where, for $\mu\in\cP_2(\mathbb{R}^d)$, $\langle \, \cdot ,\, \cdot \rangle_{\mathrm{T}_{\mu}\mathcal{P}_2(\mathbb{R}^d)}$ and $\| \cdot\|_{\mathrm{T}_{\mu}\mathcal{P}_2(\mathbb{R}^d)}$ denote the $L^2(\mu)$ inner product and norm respectively. An energy functional $\cE$ satsfying McCann's condition has the typical form
\begin{equation}\label{eq:McCann_intro}
\cE(\mu) =  \int U(\mu)\,\dd\mathcal{L}^d + \int V\,\dd\mu + \int W * \mu\,\dd\mu,
\end{equation}
where $\mathcal{L}^d$ is the Lebesgue measure on $\mathbb{R}^d$. In the above and in the rest of this article with a slight abuse of notation we shall make no distinction between an absolutely continuous probability measure $\mu$ and its density against the Lebsgue measure. The first term in the the decomposition \eqref{eq:McCann_intro} is known as internal energy: notable examples include the Boltzmann entropy $U(r)= r\log r$ and R\'eny entropies $U(r) = \frac{1}{\alpha-1}r^{\alpha}$. The second term is known as  potential energy, while the third term is referred to as the interaction energy, with $W$ often being called the interaction potential. A formal calculation suggests \cite[Sec 3.3.2]{AmGi13} that the Wasserstein gradient of $\cE$ at $\mu$ can be identified with the vector field $\otgrad\cE(\mu)$ given by
\begin{equation*}
\otgrad\cE(\mu)(x) = U''(\mu(x))\, \nabla \mu(x) + \nabla V(x) + \nabla W * \mu(x).
\end{equation*}
In light of this,  the Wasserstein gradient flow of $\cE$ started at $\mu$ is the curve $(\mu_t)_{t\geq0}$, solution of the evolution equation
\begin{equation*}
\quad \partial_t\mu_t -\frac{1}{2}\Delta P(\mu_t) + \nabla \cdot((-\nabla V -\nabla W * \mu_t )\mu_t) = 0, \quad \mu_0=\mu,
\end{equation*}
where $P(r) = rU'(r) - U(r)$. The interpretation of non linear PDEs in terms of Wasserstein gradient flows is a milestone result that has motivated much of the recent progresses in the theory of gradient flows and optimal transport, we refer to the seminal papers \cite{JoKiOt98,Ot01} and to the monograph \cite{AmGiSa08} for more information. Building on the above considerations and profiting from the formal Riemannian structure of $\cP_2(\mathbb{R}^d)$ induced by the Wasserstein distance, it is natural to formally interpret \eqref{eq:HJ_intro} as the Hamilton-Jacobi equation characterising the value function of the controlled gradient flow problem 
\begin{equation}\label{eq:control_pb_intro}
\begin{split}
\sup_{(\mu_{t},u_{t})_{t\geq0}}\,\,\,& \int_0^{+\infty} e^{-\lambda^{-1} t}\Big(-\frac12 \|u_t\|^2_{\mathrm{T}_{\mu_t}\cP_2(\mathbb{R}^d)} + \lambda^{-1}h(\mu_t)\Big) \dd t\\
& \text{s.t.} \quad \partial_t\mu_t -\frac{1}{2}\Delta P(\mu_t) + \nabla \cdot((-\nabla V -\nabla W * \mu_t + u_t)\mu_t) = 0,
\end{split}
\end{equation}
where $u_t$ is a vector field acting as the control variable, $\|u_t\|^2_{\mathrm{T}_{\mu_t}\cP_2(\mathbb{R}^d)}$ is the cost for steering the gradient flow and $h$ is a reward function. In this paper we make this interpretation rigorous by showing that a mathematically sound formulation of the control problem above provides with viscosity solutions for \eqref{eq:HJ_intro}. To do so, we rely on the notion of viscosity solution introduced in \cite{CoKrTo21},  whose formulation is in terms of rigorous upper and lower bounds for the formal Hamiltonian appearing in \eqref{eq:HJ_intro}. The definition of these operators is inspired by geodesic convexity properties of the functional $\cE$, captured by the celebrated Evolutional Variational Inequality (EVI) characterization of the corresponding gradient flow \cite{AmGiSa08}. One of the contributions of this work is to introduce and profit from a new version of EVI, that we name \emph{controlled EVI} (see Proposition \ref{prop:modif_EVI} below) which is well suited to the study of controlled gradient flows. Special instances of the abstract equation \eqref{eq:HJ_intro} have been studied over the past two decades by Feng and coauthors, see \cite{Fe06,FeKa09,FeMiZi21} for a sample of the most relevant contributions. To be more precise, the equations considered in these works correspond sometimes to an even more general version of \eqref{eq:HJ_intro} in which the underlying geometry is not necessarily the Wasserstein geometry. All uniqueness results we obtained \cite{CoKrTo21,CoKrTo23} apply to this broader class; we leave it to future work the development of  an abstract existence theory that applies at the same level of generality.
\paragraph{McKean Vlasov control and controlled gradient flows} When the internal energy in \eqref{eq:McCann_intro} is given by the Boltzmann entropy, i.e. when $U(r)=r\log r$, 
\eqref{eq:control_pb_intro} can be cast as a mean field, or McKean-Vlasov, control problem. That is to say, as a stochastic control problem in which the controlled state process is a non linear diffusion in the sense of McKean and the objective function may exhibit a non linear dependence in the law of the controlled state. Mean field control is an autonomous and thriving research field; we refer the interested reader to the monographs \cite{CaDe18a,CaDe18b} for more information and further references. To be slightly more precise, there is formal equivalence between \eqref{eq:control_pb_intro} and
\begin{equation}\label{eq:McKean_intro} 
\begin{split}
\sup_{(\alpha_t)_{t\geq0}} &\mathbb{E}\Big[ \int_0^{+\infty} e^{-\lambda t }\big(-\frac{1}{2}|\alpha_t|^2 +\lambda^{-1} h(X^{\alpha}_t)\big)\,\dd t \Big]\\
& \text{s.t.} \quad \dd X^{\alpha}_t= [-\nabla V(X^{\alpha}_t)-\nabla W *\mathrm{Law}(X^{\alpha}_t)+\alpha_t]\dd t + \sqrt{2}\dd B_t, \quad X^{\alpha}_{0}\sim \mu.
\end{split}
\end{equation}
In the above, $(B_t)_{t\geq0}$ is a Brownian motion and the state dynamics $(X^{\alpha}_t)_{t\geq0}$ as well as the control $(\alpha_t)_{t\geq0}$ are stochastic processes. Roughly speaking, the equivalence between the two problems can be explained observing that the stochastic processes considered in \eqref{eq:McKean_intro} provide with probabilistic representations for admissible curves in \eqref{eq:control_pb_intro}. Motivated by applications and the connections with mean field game theory \cite{CaDeLaLi19}, the study of Hamilton-Jacobi equations stemming from McKean-Vlasov control problems has been driving a large part of the recent research on infinite dimensional Hamilton-Jacobi equations (see for example \cite{cosso2023master,daudin2023comparison,soner2023viscosity}) and the results of our article contribute to this research line: we shall clarify below how our findings compare to the most recent results in the field.

\paragraph{Literature review} The two previous papers of this series provide with an extended overview of the different notions of solutions recently proposed to tackle Hamilton-Jacobi (henceforth HJ) equations on the space of probability measures and of the myriad of techniques developed to prove existence and uniqueness of viscosity solutions. We refer the interested reader to the introductory sections of \cite{CoKrTo21}\cite{CoKrTo23} for a compilation of references and detailed comparisons between the class of HJ equations studied here and other families of equations already considered in the literature. However, further progress has been made after our last work appeared and we shall now report on it. Some of the most recent results are concerned with equations for probability measures defined on compact spaces, a situation in which many of the main difficulties we have to deal with here do not appear. Three relevant contributions in this direction are \cite{SoYa22},\cite{Bert2023} and \cite{cecchin2022weak}. The first article studies McKean-Vlasov control problems on the torus, whereas \cite{Bert2023} introduces and  solves a stochastic version of the classical (deterministic) optimal transport problem in which the target measure evolves according to a stochastic process. The source of noise in the resulting HJ equation is then not in the dynamic of each agent, as it is the case in most works on the subject, but rather in the dynamics of the target measure. In \cite{cecchin2022weak} HJ equations for probability measures on the torus are solved and uniqueness is proved among suitable classes of semiconcave functions. In particular, this work does not appeal to viscosity solutions and relies on Fourier analysis to obtain finite-dimensional approximations of the target equation. The  most recent contributions about classes of HJ equations that are more closely related to the one we consider here are \cite{cosso2023master},\cite{soner2023viscosity} and \cite{daudin2023comparison} all look at equations arising in the context of mean field control. The authors of \cite{cosso2023master} show well posedness by analyzing finite-dimensional projections of the equation associated with the $N$-agents approximation of the McKean-Vlasov control problem at hand. In doing so, they profit from a lifting of the Wasserstein space to a suitably defined space of square integrable random variables, which has the convenient property of being an Hilbert space. Deviating from this approach, \cite{soner2023viscosity} relies on Fourier representations of Sobolev norms on the space of probability measures and a convenient representation of the derivatives of Wasserstein Lipschitz functions to establish the comparison principle. The article \cite{daudin2023comparison} adopts yet another approach to the design of a proper intrinsic notion of viscosity solution based on Wasserstein subdifferentials. Finally, we mention the very recent work \cite{daudin2023well}. In there, the authors succeed in establishing well posedness for a class of semilinear PDEs on the space of probability measures on the torus, allowing for Hamiltonians that are not convex in the momentum variable and for the presence a common noise. To do so, they exploit a weak metric previously introduced in \cite{bayraktar2023comparison} and first show a partial comparison principle for solutions that are Lipschitz with respect to the weak metric. Then, they show that under suitable growth and regularity assumptions on the coefficient, solutions with the desired Lipschitz properties actually do exist.
\paragraph{Our contribution} In \cite{CoKrTo21}, inspired by some heuristic geometric considerations about geodesically convex functionals, we defined rigorous upper and lower bounds for the formal Hamiltonian associated to \eqref{eq:HJ_intro} that do not require to manipulate any notion of derivative or subdifferential in the Wasserstein space. Then, we defined viscosity solutions in terms of these operators and showed a comparison principle. The subsequent paper \cite{CoKrTo23} lays the foundations for the existence theory we develop in this work by showing that the comparison principle of \cite{CoKrTo21} implies a comparison principle for a new notion of viscosity solutions that makes use of a simpler and more tractable set of test functions. These are simple cylindrical test functions of the square Wasserstein distance, whereas the test functions in \cite{CoKrTo23} contained the Tataru's distance \cite{Ta94} that is not smooth. When comparing the assumptions required for our theory to apply and those in the above mentioned works, one can observe that all these works deal exclusively with McKean-Vlasov control problem. Therefore, they are unable to treat situations in which the internal energy is different form the Boltzmann entropy, leaving out R\'eny entropies. In this setting, our results appear to be genuinely new, at least to the best of our knowledge and understanding. If we remain in the realm of McKean-Vlasov control problems, the above mentioned works are generally more flexible concerning the structure of the Hamiltonian. For example, neither \cite{cosso2023master} nor \cite{daudin2023comparison} require the Lagrangian to be separable or the cost to be quadratic in the control variable. Furthermore, they encompass situations in which the diffusion coefficient in the controlled dynamics is not a constant matrix. In turn, stronger regularity assumptions are imposed on the coefficients in all the above mentioned works . More precisely, some  form of Wasserstein Lipschitzianity on either the Hamiltonian or the cost functional is assumed there, whereas all we require here is that the running cost is bounded and continuous in some $p-$Wasserstein topology for $p<2$. To conclude, we wish to point out that even though the solution theory we build in this work applies to controlled Wasserstein gradient flows only, the derivative-free approach to uniqueness of \cite{CoKrTo21,CoKrTo23} applies to a more general class of equations, set on metric spaces that may well differ from the Wasserstein space. Some candidate equations for which our uniqueness results could in principle be applied have been described in the introduction of \cite{CoKrTo21}. The question of existence, left open in this work, calls for further investigations.
}

\paragraph{Organization} The manuscript is organized as follows. In Section \ref{section:settingAssumptionsMainresults} we introduce the notion of viscosity solution we proposed in our earlier works, state our main hypothesis and define our candidate solution through an abstract controlled gradient flow problem. In Section \ref{section:preliminaryResults} we focus on obtaining estimates for controlled gradient flows, in particular by establishing a controlled version of EVI. Sections \ref{sec:subsolution} and \ref{sec:supersolution} are devoted to establishing the subsolution and supersolution property. Some technical results are gathered in Appendix \ref{section:appendix}.

\section{Setting, assumptions and main results} \label{section:settingAssumptionsMainresults}

\textbf{Frequently used notation}. We write $C(\cP_2(\R^d)), LSC(\cP_2(\R^d))$, and $USC(\cP_2(\R^d))$ for the spaces of continuous, lower semi-continuous and upper semi-continuous functions from $\cP_2(\R^d)$ into $\bR$. We denote by $C_u(\cP_2(\R^d)), C_l(\cP_2(\R^d)), LSC_l(\cP_2(\R^d))$ and $USC_u(\cP_2(\R^d))$ the subsets of functions that admit a lower or upper bound. Finally $C_b(\cP_2(\R^d)) = C_u(\cP_2(\R^d)) \cap C_l(\cP_2(\R^d))$. The space of smooth real functions with compact support is denoted by $\cC^\infty_c([0,T]\times \R^d)$, while the space of  absolutely continuous functions whose derivatives belong to $L^2([0,T])$ is denoted with $AC^2([0,T], \cP_2(\R^d))$.

We proceed in Section \ref{subsection:SettingAssumptions} with our setting and main assumptions. In Section \ref{subsection:CandidateSolution} we introduce our candidate solution. In Section \ref{subsection:UpperLowerHamiltonian} we introduce our rigorous upper and lower Hamiltonian. Finally, in Section \ref{subsection:MainResults} we state our main results.

\subsection{Setting and assumptions} \label{subsection:SettingAssumptions}

In this section we recall all the relevant tools used in the article.

We start with the definition of Wasserstein distance and Wasserstein space. In the space of probability measures with finite second moment $\cP_2(\R^d)$, we introduce the  Kantorovich-Rubinstein-Wasserstein distance  of order two $W_2$, defined by
\begin{equation*}
W^2_2(\mu,\nu) = \inf_{\pi\in\Pi(\mu,\nu)} \int |x-y|^2 \pi(\De x \De y),
\end{equation*}
where $\Pi(\mu,\nu)$ is the set of $2$-plans with given marginals $\mu,\nu$. The space   $(\cP_2(\R^d), W_2)$ is called the Wasserstein space.

Throughout the whole manuscript we are interested in entropy functionals  on the Wasserstein space $\cE: \mathcal{P}_2(\R^d)\to (-\infty, +\infty]$  of the following type:

\begin{equation}\label{eq:def-functional2}
\cE(\mu) := \int U(\rho)\,\dd\cL^d + U'(\infty)\mu^\perp(\R^d) + \int V(x)\,\dd\mu(x) + \int W(x-y)\,\dd\mu(x)\dd\mu(y),
\end{equation}
where $\mu = \rho\cL^d + \mu^\perp$, $\mu^\perp \perp \cL^d$ and $U'(\infty) := \lim_{r \to \infty} U'(r)$. Note that if $U$ has superlinear growth at infinity, as we will require in the assumptions below, see Assumptions \ref{ass: OT energy functional}, then the above definition of $\cE$ reduces to
\[
\cE(\mu) = \left\{ \begin{array}{ll}
\displaystyle{\int U(\mu)\,\dd\cL^d + \int V\,\dd\mu + \int W * \mu\,\dd\mu} & \qquad \textrm{if } \mu \ll\cL^d, \\
+\infty & \qquad \textrm{otherwise}.
\end{array}
\right.
\]
Note that, as stated in the introduction, with a slight abuse of notation we shall make no distinction between an absolutely continuous probability measure $\mu$ and its density against the Lebsgue measure.

Moreover, the former is the lower semicontinuous relaxation (w.r.t.\ the narrow topology) of the latter, and it can be checked that under a superlinearity assumption $\cE$ is actually lower semicontinuous w.r.t.\ the weak $L^1$ topology.

For sake of brevity, let us introduce the following notation
\[
\begin{split}
\mathcal U(\mu) & := \int U(\rho)\,\dd\cL^d + U'(\infty)\mu^\perp(\R^d), \\
\mathcal V(\mu) & := \int V(x)\,\dd\mu(x), \\
\mathcal W(\mu) & := \int W(x-y)\,\dd\mu(x)\dd\mu(y).
\end{split}
\]
On the functionals internal energy $\mathcal U$, potential energy $\mathcal V$, and interaction energy $\mathcal W$ we formulate the following assumptions:

\begin{assumption}[McCann's condition]\label{ass: OT energy functional}
\begin{enumerate}[(a)] 
\item $U:[0,+\infty)\rightarrow \R$ is convex, differentiable with superlinear growth. It satisfies the doubling condition
\[
\exists C>0: \quad U(z+w) \leq C(1+U(z)+U(w)), \quad \forall z,w\geq 0. 
\] 
Moreover we assume that 
\[ 
s\mapsto s^dU(s^{-d}) \quad \text{is convex and non-increasing on $(0,+\infty)$}
\]
and 
\[ 
U(0)=0, \quad \lim_{s\rightarrow 0} U(s)/s^{-\alpha}>-\infty, \quad \text{for some $\alpha>\frac{d}{d+2}$}.
\]
\item $V:\R^d\rightarrow (-\infty,+\infty]$ is lower semi-continuous, $\kappa_V$-convex for some $\kappa_V\in \R$, with proper domain that has nonempty interior.
\item $W:\R^d\rightarrow[0,\infty)$ is an even continuously differentiable $\kappa_W$-convex function for some $\kappa_W \in \R$
and satisfies the doubling condition 
\[ 
\exists C>0: \quad W(x+y) \leq C(1+W(x)+W(y)), \quad \forall x,y\in \R^d.
\]
\end{enumerate}
\end{assumption}


\begin{remark}\label{rmk:lsc}
The functional $\cE$ is lower semicontinuous w.r.t.\ $W_p$ convergence, for any $p \in [1,2]$. Indeed, the standing assumptions on $U$ grant that $\cU$ is actually narrowly/weakly lower semicontinuous, as discussed in \cite[Section 10.4.3]{AmGiSa08}. Secondly, the $W_p$-lower semicontinuity of $\cV$ follows from \cite[Example 9.3.1]{AmGiSa08}, since the $\kappa_V$-convexity of $V$ implies that the negative part of $V$ has 2-growth (i.e.\ $V(x) \geq -A-B|x|^2$ for all $x \in \R^d$ and for some $A,B \in \R$), so that \emph{a fortiori} $V$ has $p$-growth for any $p \in [1,2]$. Finally, $\cW$ is $W_p$-lower semicontinuous too, since the $\kappa_W$-convexity of $W$ implies a $p$-growth condition on $W^-$, for any $p \in [1,2]$, and this implies the desired lower semicontinuity property, as discussed in \cite[Example 9.3.4]{AmGiSa08}.
\end{remark}

We also introduce the information functional $\cI: \cP_2(\R^d) \to [0, +\infty]$ as
\[
\cI(\pi) : = \left\{
\begin{array}{cc}
|\partial \cE|^2(\pi)  & \textrm{if } \pi\in \cD(\cE), \\
+\infty  & \text{otherwise},
\end{array} \right.,\
\]
where $|\partial\cE|$ denotes the local slope of $\cE$, defined as
\[
|\partial\cE|(\pi) := \limsup_{\mu \to \pi}\frac{(\cE(\pi)-\cE(\mu))^+}{W_2(\pi,\mu)}.
\]
It is well known that, under Assumption \ref{ass: OT energy functional}, $\cE$ is geodesically $\kappa$-convex, for $\kappa:=\kappa_V + \kappa_W$ (see e.g.\ \cite[Section 9.3]{AmGiSa08}). Let us also remark that all measures in $\cD(\cI)$ are regular in the sense of \cite[Definition 6.2.2]{AmGiSa08}, since in $\cP_2(\R^d)$, regularity boils down to absolute continuity w.r.t.\ the Lebesgue measure $\cL^d$. For these reasons (see \cite[Section 10.1.1]{AmGiSa08}), a vector $v \in L^2(\mu)$ belongs to the subdifferential of $\cE$ at $\mu \in \cD(\cI)$ if and only if
\begin{equation}\label{eq:subdifferential}
\cE(\nu)-\cE(\mu) \geq \int \ip{v(x)}{\otmap{\mu}{\nu}(x)-x}\dd\mu(x) + \frac{\kappa}{2}W_2^2(\mu,\nu), \qquad \forall \nu \in \cD(\cE),
\end{equation}
where $\otmap{\mu}{\nu}$ denotes the optimal transport map that pushes $\mu$ onto $\nu$ (whose existence and uniqueness is ensured by \cite[Theorem 6.2.4]{AmGiSa08} when $\mu \ll \cL^d$; this is satisfied as soon as $\mu\in\cD(\cE)$).
		 


\subsection{The candidate solution} \label{subsection:CandidateSolution}

Let us first of all introduce the notion of \emph{admissible curve}.

\begin{definition}\label{def:admissible}
We say that a pair $(\mu_{t},u_{t})_{t\geq0} \in \admt$, also written $(\mu_{\cdot},u_{\cdot}) \in \admt$ if and only if:
\begin{itemize}
\item $\mu. \in AC^2([0,T],\cP_2(\R^d))$ for every $T > 0$, namely
\[
\exists \lim_{h \to 0}\frac{W_2(\mu_{t+h},\mu_t)}{|h|} =: |\dot{\mu}_t|, \quad \textrm{for a.e. } t \in [0,T] \qquad \textrm{and} \qquad |\dot{\mu}_t| \in L^2(0,T).
\]
\item We have
\begin{equation}\label{eq:integrability-control}
\int_{0}^{T}\|u_t\|^2_{L^2(\mu_t)} \dd t<+\infty.
\end{equation}
\item $u_{\cdot}\in H^{-1}(\mu.)$, where 
\[
H^{-1}(\mu.)= \overline{\{(t,x)\mapsto\nabla\psi_t(x):\psi\in \cC_{c}^{\infty}([0,T] \times \R^d) \}}^{L^2(\dd t \otimes \dd\mu_t)},
\]
\item $(\mu_{\cdot},u_{\cdot})$ solves 
\begin{equation}\label{eq:controlled_GF}
\partial_t\mu_t -\frac{1}{2}\Delta P(\mu_t) + \nabla \cdot((-\nabla V -\nabla W * \mu_t + u_t)\mu_t) = 0, 
\end{equation}
for $(t,x)\in[0,T]\times\R^d$, in the sense of distributions, where $P(r) := rU'(r) - U(r)$.
\end{itemize}
Note that $\admt$ is nonempty because the constant pair $(\mu, 0)$ solves the equation \eqref{eq:controlled_GF} for any $\mu\in\cP_2(\R^d)$.

We say that a pair $(\mu_{t},u_{t})_{t\geq0} \in \admerg$ if its restriction to $[0,T]$ belongs to $\admt$ for all $T$ and 
\[
\int_0^\infty e^{-\lambda^{-1}t} \|u_t\|^2_{L^2(\mu_t)} dt < +\infty.
\]
Note that according to its definition, the set $\admerg$ depends on $\lambda$.
\end{definition}

\begin{remark}\label{rmk:continuity_equation}
Recall that, by \cite[Theorem 8.3.1]{AmGiSa08}, $\mu. \in AC^2([0,T],\cP_2(\R^d))$ if and only if there exists a Borel vector field $(v_t)_{t \in [0,T]}$ such that $\|v_t\|_{L^2(\mu_t)} \in L^1(0,T)$ and $(\mu_t,v_t)$ is a distributional solution of the continuity equation $\partial_t\mu_t + \nabla\cdot(v_t\mu_t)=0$. In this case, $\|v_t\|_{L^2(\mu_t)} = |\dot{\mu}_t|^2$ for a.e.\ $t \in [0,T]$.
\end{remark}

We define the candidate solution $\Phi : \cP_2(\R^d) \to \mathbb{R}$ as the value function of the control problem
\begin{equation}\label{def:value_fun}
\begin{split}
& \Phi(\mu) := \sup_{\substack{(\mu_{\cdot},u_{\cdot}) \in \admerg \\ \mu_0=\mu}} \cA(\mu_{\cdot},u_{\cdot}), \\
& \textrm{where for } (\mu_{\cdot},u_{\cdot}) \in \admerg \textrm{ with } \mu_0=\mu,  \\
& \cA(\mu_{\cdot},u_{\cdot}) := \int_{0}^{\infty} e^{-\lambda^{-1}t}\Big(-\frac12\|u_t\|^2_{L^2(\mu_t)} + {\lambda^{-1}} h(\mu_t)\Big) \dd t.
\end{split}
\end{equation}
For future convenience, let us also define for any $0 \leq t \leq T \leq +\infty$ the restricted action
\begin{equation}\label{eq:action_T}
\cA_{t,T}(\mu_{\cdot},u_{\cdot}) := \int_t^T e^{-\lambda^{-1}s}\Big(-\frac12 \|u_s\|^2_{L^2(\mu_s)} + {\lambda^{-1}} h(\mu_s)\Big) \dd s.
\end{equation}
When $t=0$, for sake of brevity we will write $\cA_T := \cA_{0,T}$.





\subsection{The rigorous upper and lower Hamiltonian} \label{subsection:UpperLowerHamiltonian}

We define now the set of Hamiltonians that rigorously upper and lower bound the formal Hamiltonian of \eqref{eq:HJ_intro}
\begin{equation*}\label{eq:HJ_formal}
H_{\text{formal}}f(\mu) =  \langle -\otgrad \cE , \otgrad f \rangle_{\mathrm{T}_{\mu}\mathcal{P}_2(\mathbb{R}^d)}  + \frac12\|\otgrad f \|^2_{\mathrm{T}_{\mu}\mathcal{P}_2(\mathbb{R}^d)}
\end{equation*}
in terms of smooth cylindrical test functions. The definition is motivated in \cite{CoKrTo23}, where for viscosity solutions of the Hamilton-Jacobi equation in terms of these Hamiltonians a comparison principle is shown. 

Let $\cT$ be the collection of functions $\varphi$ defined as 
\begin{equation} \label{eqn:defT}
\cT := \left\{ \varphi \in \mathcal{C}^{\infty}([0,\infty)^{k};\mathbb{R}) \,:\, k\in\mathbb{N},\,\forall \, i=1,\dots, k  \,\, \partial_i \varphi >  0  \, \right\}, 
\end{equation}
where $\cC^\infty([0,\infty)^{k};\mathbb{R})$ is the set of smooth functions mapping $[0,\infty)^{k}$ into $\bR$. 
For $\mu_1,\dots,\mu_k \in\mathcal{P}_2(\R^d)$, we write  $\bm{\mu} = (\mu_1,\dots,\mu_k)$ and $\bm{\mu} \in {\cD(\cI) }$  if all elements in the vector are in $\cD(\cI)$. Moreover $W_2 (\cdot,\bm\mu)= (W_2(\cdot,\mu_1),\dots, W_2(\cdot, \mu_k))$, $\cE(\bm\mu)=(\cE(\mu_1),...,\cE(\mu_k))$, and $\bm{1}=(1,...,1)$. We shall also denote by $\cdot$ the Euclidean inner product.

In the following, we set  $\kappa := \kappa_V+\kappa_W$ the sum of the convexity constants defined in Assumption \ref{ass: OT energy functional}.

We next introduce the formal upper and lower bound 
in terms of smooth cylindrical test functions. The upper bounds are derived based on 

\begin{definition} \label{definition:H}
For $a > 0$, $\varphi \in \cT$, $\rho \in \cD(\cI)$, and $\bm\mu = (\mu_1,\ldots,\mu_k) \in (\cP_2(\R^d))^{k}$ such that $\bm{\mu} \in \cD(\cI)$, we define $f^\dagger = f^\dagger_{a,\varphi,\rho,\bm\mu} \in C_l(\cP_2(\R^d))$ and $g^\dagger = g^\dagger_{a,\varphi,\rho,\bm\mu} \in USC(\cP_2(\R^d))$ for all $\pi \in \cP_2(\R^d)$ as
\begin{align}
f^\dagger(\pi) & := \frac{a}{2} W_2^2(\pi,\rho) + \varphi\left( \frac{1}{2}W_2^2(\pi,\bm\mu)\right), \label{eq:reg_f_1dag} \\
g^\dagger(\pi) & := a \left[\cE(\rho) - \cE(\pi) - \frac{\kappa}{2} W_2^2(\pi,\rho) \right] + \frac{a^2}{2} W_2^2(\pi,\rho) \label{eq:reg_g_1dag}\\
& \qquad +\sum_{i=1}^k \partial_i \varphi\left( \frac{1}{2}W_2^2(\pi,\bm\mu)\right) \left[\cE(\mu_i) - \cE(\pi) - \frac{\kappa}{2} W_2^2(\pi,\mu_i) \right]  \notag \\
& \qquad + \frac{1}{2} \left( \sum_{i = 1}^k \partial_i \varphi\left( \frac{1}{2}W_2^2(\pi,\bm\mu)\right) W_2(\pi,\mu_i) \right)^2 \notag \\
& \qquad + a W_2(\pi,\rho) \left(\sum_{i = 1}^k \partial_i \varphi\left( \frac{1}{2}W_2^2(\pi,\bm\mu)\right) W_2(\pi,\mu_i)\right) \notag
\end{align}
and set $H_\dagger \subseteq C_l(\cP_2(\R^d)) \times USC(\cP_2(\R^d))$ by
\[
H_{\dagger} := \left\{ (f^{\dagger},g^{\dagger}) \,:\, \forall \varphi \in \cT,\, a > 0,\, \rho \in \cD(\cI),\, \bm\mu\in {\cD(\cI) } \right\}.
\]
In the same way, for $a > 0$, $\varphi \in \cT$, $\gamma \in \cD(\cI)$ and $\bm\pi=(\pi_1,\ldots,\pi_k) \in (\cP_2(\R^d))^k$ such that $\bm\pi \in {\cD(\cI)}$ we define $f^\ddagger = f^\ddagger_{a,\varphi,\gamma,\bm\pi} \in C_u(\cP_2(\R^d))$ and $g^\ddagger = g^\ddagger_{a,\varphi,\gamma,\bm\pi} \in LSC(\cP_2(\R^d))$ for all $\mu \in \cP_2(\R^d)$ as
\begin{align}
f^\ddagger(\mu) & := - \frac{a}{2} W_2^2(\mu,\gamma) -\varphi\left( \frac{1}{2}W_2^2(\mu,\bm\pi)\right), \label{eq:reg_f_2ddag} \\
g^{\ddagger}(\mu) & := a \left[\cE(\mu) -\cE(\gamma) + \frac{\kappa}{2}W_2^2(\mu,\gamma) \right] + \frac{a^2}{2} W_2^2(\gamma,\mu) \label{eq:reg_g_2ddag} \\
& \qquad + \sum_{i=1}^k \partial_i \varphi\left(\frac{1}{2}W_2^2(\mu,\bm\pi)\right)\left[\cE(\mu) -\cE(\pi_i) + \frac{\kappa}{2}W_2^2(\mu,\pi_i) \right] \notag \\
& \qquad - \frac{1}{2}  \left(\sum_{i=1}^k  \partial_i \varphi\left( \frac{1}{2}W_2^2(\mu,\bm\pi)\right) W_2(\mu,\pi_i)\right)^2 \notag \\
& \qquad - a W_2(\mu,\gamma)\left(\sum_{i=1}^k  \partial_i \varphi\left( \frac{1}{2}W_2^2(\mu,\bm\pi)\right) W_2(\mu,\pi_i)\right) \notag
\end{align}
and set $H_\ddagger \subseteq C_u(\cP_2(\R^d)) \times LSC(\cP_2(\R^d))$ by
\[
H_\ddagger := \left\{ (f^\ddagger,g^\ddagger) \,:\, \forall \varphi \in \cT,\, a >0,\, \gamma \in \cD(\cI),\, \bm\pi \in {\cD(\cI) } \right\}.
\]
\end{definition}

As the definitions of $g^\dagger,g^\ddagger$ are particularly involved, to avoid cumbersome computations we will often use the following alternative (more compact) notations:
\[
\begin{split}
g^\dagger(\pi) & = a \Big(\cE(\rho) - \cE(\pi) - \frac{\kappa}{2} W_2^2(\pi,\rho) \Big) \\
& \qquad + \nabla\varphi\bigg(\frac12 W_2^2(\pi,\bm\mu) \bigg) \cdot \Big(\cE(\bm\mu) - \cE(\pi)\bm{1} - \frac{\kappa}{2} W_2^2(\pi,\bm\mu)\Big) \\
& \qquad + \frac{1}{2} \left( a W_2(\pi,\rho) + \nabla\varphi(\frac12W_2^2(\pi,\bm\mu))\cdot W_2(\pi,\bm\mu)\right)^2\,, \\
g^\ddagger(\mu) & = a \Big(\cE(\mu) - \cE(\gamma) + \frac{\kappa}{2} W_2^2(\mu,\gamma) \Big) \\
& \qquad + \nabla\varphi\bigg(\frac12 W_2^2(\mu,\bm\pi) \bigg) \cdot \Big(\cE(\mu)\bm{1} - \cE(\bm\pi) + \frac{\kappa}{2} W_2^2(\mu,\bm\pi)\Big) \\
& \qquad + a^2 W_2^2(\gamma,\mu) - \frac{1}{2} \left( a W_2(\gamma,\mu) + \nabla\varphi\left(\frac12W_2^2(\mu,\bm\pi)\right)\cdot W_2(\mu,\bm\pi)\right)^2 \,.
\end{split}
\]

\subsection{Main results} \label{subsection:MainResults}

We are now ready to make precise the notion of solution we are looking for. The notion of viscosity solution used in this article comes from the one given in \cite{CoKrTo23}, where it is used to prove uniqueness.  We will state this definition for  general Hamiltonians  {$A_\dagger, \widehat{A}_\dagger \subseteq LSC_l(\mathcal{P}_2(\R^d)) \times USC(\mathcal{P}_2(\R^d))$ and $A_\ddagger, \widehat{A}_\ddagger \subseteq USC_u(\mathcal{P}_2(\R^d)) \times LSC(\mathcal{P}_2(\R^d))$. }


\begin{definition} \label{definition:viscosity_solutions_HJ_sequences}
Fix $\lambda > 0$ and $h^\dagger,h^\ddagger \in C_b(\mathcal{P}_2(\R^d))$. Consider the equations
\begin{align} 
f - \lambda  A_\dagger f & = h^\dagger, \label{eqn:differential_equation_tildeH1} \\
f - \lambda A_\ddagger f & = h^\ddagger. \label{eqn:differential_equation_tildeH2}
\end{align}
We say that $u$ is a \textit{(viscosity) subsolution} of equation \eqref{eqn:differential_equation_tildeH1} if $u$ is bounded, upper semi-continuous and if for all $(f,g) \in A_\dagger$ there exists a sequence $(\pi_n)_{n\in \mathbb N} \subseteq \cP_2(\R^d)$ such that
\begin{gather}
\lim_{n \uparrow \infty} u(\pi_n) - f(\pi_n)  = \sup_\pi u(\pi) - f(\pi), \label{eqn:viscsub1} \\
\limsup_{n \uparrow \infty} u(\pi_n) - \lambda g(\pi_n) - h^\dagger(\pi_n) \leq 0. \label{eqn:viscsub2}
\end{gather}
We say that $v$ is a \textit{(viscosity) supersolution} of equation \eqref{eqn:differential_equation_tildeH2} if $v$ is bounded, lower semi-continuous and if for all $(f,g) \in A_\ddagger$ there exists a sequence $(\pi_n)_{n\in \mathbb N} \subseteq \cP_2(\R^d)$ such that
\begin{gather}
\lim_{n \uparrow \infty} v(\pi_n) - f(\pi_n)  = \inf_\pi v(\pi) - f(\pi), \label{eqn:viscsup1} \\
\liminf_{n \uparrow \infty} v(\pi_n) - \lambda g(\pi_n) - h^\ddagger(\pi_n) \geq 0. \label{eqn:viscsup2} 
\end{gather}
If $h^\dagger = h^\ddagger$, we say that $u$ is a \textit{(viscosity) solution} of equations \eqref{eqn:differential_equation_tildeH1} and \eqref{eqn:differential_equation_tildeH2} if it is both a subsolution of \eqref{eqn:differential_equation_tildeH1} and a supersolution of \eqref{eqn:differential_equation_tildeH2}.
		
We say that \eqref{eqn:differential_equation_tildeH1} and \eqref{eqn:differential_equation_tildeH2} satisfy the \textit{comparison principle} if for every subsolution $u$ to \eqref{eqn:differential_equation_tildeH1} and supersolution $v$ to \eqref{eqn:differential_equation_tildeH2}, we have $\sup_{\mathcal{P}_2(\R^d)} u-v \leq \sup_{\mathcal{P}_2(\R^d)} h^\dagger - h^\ddagger$.
\end{definition}







\begin{theorem}
Under Assumption \ref{ass: OT energy functional}, let $h \in C_b(\cP_p(\R^d))$ for some $p<2$. Let  $\Phi^*$ and $\Phi_*$ be respectively the upper and the lower semicontinuous relaxation, w.r.t.\ the $W_p$-topology, of the value function $\Phi$ defined in \eqref{def:value_fun}. Then $\Phi^*$ is a viscosity subsolution of the Hamilton-Jacobi equation \eqref{eqn:differential_equation_tildeH1} and $\Phi_*$ a viscosity supersolution of \eqref{eqn:differential_equation_tildeH2} where $h^\dag = h^\ddag = h$, $A_\dag = H_\dag$, and $A_\ddag = H_\ddag$ as defined in Definition \ref{definition:H}.
\end{theorem}

The proof of the above theorem is given in Sections \ref{sec:subsolution} and \ref{sec:supersolution}. 

Applying now the comparison principle stated in \cite[Corollary 3.17]{CoKrTo23} to $\Phi^*$ and $\Phi_*$, we obtain the following corollary.

\begin{corollary}
Under Assumption \ref{ass: OT energy functional}, let $h \in C_b(\mathcal P_p(\R^d))$ for some $p<2$  be weakly continuous. Then the value function $\Phi$ defined in \eqref{def:value_fun} is the unique  viscosity solution of the Hamilton-Jacobi equations \eqref{eqn:differential_equation_tildeH1} and \eqref{eqn:differential_equation_tildeH2}, where $h^\dag = h^\ddag=h$, $A_\dag = H_\dag$, and $A_\ddag = H_\ddag$ as defined in Definition \ref{definition:H}. 
\end{corollary}
 
\section{Preliminary results} \label{section:preliminaryResults}

In this section, we start out with establishing the dynamic programming principle in Section \ref{subsection:DPP}. We proceed with introducing a modified version of EVI in Section \ref{subsection:controlledEVI}, and close off in Section \ref{subsection:continuityTransportMap} by establishing a continuity property for the optimal transport map.

\subsection{The dynamic programming principle} \label{subsection:DPP}

We first show that the value function $\Phi$ introduced in \eqref{def:value_fun} complies with the Dynamic Programming Principle (DPP). The proof of the DPP given here is a generalization of the classical finite dimensional proof that can be found for example in \cite{BCDbook}. For other DPP proofs in the infinite dimensional case, that can apply  to some of the cases covered in our context, we refer to \cite{DPT22}, \cite{cosso2023master}.  Please note that, from now on, to ease the notation, we are setting $\lambda=1$. 

\begin{proposition}[Dynamic Programming Principle] \label{proposition:DPP}
For all $\mu\in\cP_2(\R^d)$ and $T>0$ we have:
\begin{equation}\label{eq:DPP}
\tag{DPP}
\Phi(\mu)=\sup_{\substack{(\mu_{\cdot},u_{\cdot}) \in \admt \\ \mu_0=\mu}} \int_{0}^{T} e^{-t}\Big(-\frac12\|u_t\|^2_{L^2(\mu_t)}+h(\mu_t)\Big)  \dd t + e^{-T}\Phi(\mu_T).
\end{equation}
\end{proposition}

\begin{proof}
Let us fix $\mu\in \mathcal P_2(\R^d)$ and $T>0$ and name $w(\mu)$ the right-hand side of \eqref{eq:DPP}. Recall that 
\[
\Phi(\mu) := \sup_{\substack{(\mu_{\cdot},u_{\cdot}) \in \admerg,\\ \mu_0=\mu}} \int_0^\infty e^{-t}\Big(-\frac12\|u_t\|^2_{L^2(\mu_t)}+h(\mu_t)\Big)  \dd t.
\]
We first show that $\Phi(\mu)\leq w(\mu)$. If $w(\mu)=+\infty$ there is nothing to prove. Otherwise, for all $(\mu_{\cdot},u_{\cdot}) \in \admerg$ with $\mu_0=\mu$ we have (recall \eqref{def:value_fun}, \eqref{eq:action_T} for the notation)
\begin{align*}
\cA(\mu_\cdot,u_\cdot) & = \cA_T(\mu_\cdot,u_\cdot) + \int_{T}^{\infty} e^{-t}\Big(-\frac12 \|u_t\|^2_{L^2(\mu_t)} + h(\mu_t)\Big) \dd t \\ 
& = \cA_T(\mu_\cdot,u_\cdot) + \int_0^\infty e^{-s-T}\Big(-\frac12 \|u_{s+T}\|^2_{L^2(\mu_{s+T})} + h(\mu_{s+T})\Big) \dd s \\
& = \cA_T(\mu_\cdot,u_\cdot) + e^{-T}\int_0^\infty e^{-t}\Big(-\frac12\|\tilde u_{t}\|^2_{L^2(\tilde \mu_{t})}+h(\tilde \mu_{t})\Big)  \dd t,
\end{align*}
where $\tilde u(t) := u(t+T)$ and $\tilde \mu(t) := \mu(t+T)$ are such that $(\tilde \mu_{\cdot},\tilde u_{\cdot}) \in \admerg, \tilde \mu_0=\mu(0+T)=\mu_T$. Taking the supremum we obtain $\Phi(\mu)\leq w(\mu)$. 

Let us now prove the opposite inequality. If $\Phi(\mu) = +\infty $ there is nothing to prove. Otherwise, for all $(\mu_{\cdot}, u_{\cdot}) \in \admt$ with $\mu_0=\mu$ and $\varepsilon>0$, let us consider $(\tilde\mu_{\cdot},\tilde u_{\cdot}) \in \admerg, \tilde \mu_0=\mu(T)$ such that
\[
\cA_{\infty}(\tilde\mu.,\tilde u.) \geq \Phi(\mu(T))-\varepsilon.
\]
Define now
\begin{equation*}
    \bar u (t)=
\begin{cases}
    u(t) & \text{if } 0 \leq t \leq T, \\
    \tilde u(t-T) & \text{if } t \geq T,
\end{cases}
\end{equation*}
and 
\begin{equation*}
    \bar \mu(t) = 
    \begin{cases}
    \mu(t) & \text{if } 0 \leq t \leq T, \\
    \tilde \mu(t-T) & \text{if } t \geq T.
    \end{cases}
\end{equation*}
Then we have $(\bar \mu_{\cdot},\bar u_{\cdot}) \in \admerg$ with $\bar \mu_0=\mu$, so that
\begin{align*}
\Phi(\mu) & \geq \cA_\infty(\bar \mu.,\bar u.) \\ 
& = \cA_T(\mu_\cdot,u_\cdot) + \int_T^\infty e^{-t}\Big(-\frac12\|\tilde u_{t-T}\|^2_{L^2(\tilde \mu_{t-T})} + h(\tilde\mu_{t-T})\Big) \dd t \\
& = \cA_T(\mu_\cdot,u_\cdot) + e^{-T} \int_{0}^{\infty}e^{-s}\Big(-\frac12\|\tilde u_{s}\|^2_{L^2(\tilde \mu_{s})} + h(\tilde \mu_{s})\Big) \dd s \\
& \geq \cA_T(\mu_\cdot,u_\cdot) + e^{-T}\big(\Phi(\mu(T)) - \eps\big).
\end{align*}
By the arbitrariness of $(\mu_{\cdot}, u_{\cdot})$ and $\eps$, we obtain $\Phi(\mu) \geq w(\mu)$.
\end{proof}

As an immediate consequence of the DPP, we can show the upper semicontinuity of the value function $\Phi$ along admissible curves.

\begin{corollary} \label{corollary:Phi_usc_admissiblecurves}
Let $h$ be bounded and $W_2$-continuous and $(\mu.,u.)\in\admerg$. Then $|\Phi| \leq \|h\|_\infty$ and $t \mapsto \Phi(\mu_t)$ is upper semicontinuous at $t = 0$.

In particular, if $\mu.$ is the gradient flow of $\cE$ starting at $\mu_0$, then
\[
\forall t\geq0  \quad \Phi(\mu_t) \leq \Phi(\mu_0) + 2 t \|h\|_\infty.
\]
\end{corollary}

\begin{proof}
The boundedness of $\Phi$ is straightforward and does not actually rely on \eqref{eq:DPP}. Indeed, using the gradient flow starting from $\mu$ as a competitor in the definition of $\Phi(\mu)$ gives $\Phi(\mu) \geq -\|h\|_\infty$. The upper bound is trivial.

As for the upper semicontinuity of $t \mapsto \Phi(\mu_t)$, by \eqref{eq:DPP} we have for all $T>0$
\[
\begin{split}
e^{-T}\Phi(\mu_T) & \leq \Phi(\mu_0) + \int_0^T e^{-t} \left(\frac{1}{2}\|u_t\|^2_{L^2(\mu_t)} - h(\mu_t) \right) \dd t \\
& = \Phi(\mu_0) + \int_0^\infty \chi_{[0,T]}(t)e^{-t} \left(\frac{1}{2}\|u_t\|^2_{L^2(\mu_t)} - h(\mu_t) \right) \dd t,
\end{split}
\]
where $\chi_{[0,T]}$ is the indicator function of the interval $[0,T]$. Observing that the integrand function converges pointwise to 0 as $T\to 0$, by \eqref{eq:integrability-control} we can use the limsup version of Fatou's lemma to find
\[
\limsup_{T \downarrow 0} \Phi(\mu_T) \leq \Phi(\mu_0).
\]
The final part of the statement follows noting that the control associated to a gradient flow trajectory is $u_t \equiv 0$ for all $t \geq 0$, so that plugging this information into \eqref{eq:DPP} yields
\[
\begin{split}
\Phi(\mu_T) & \leq (1-e^{-T})\Phi(\mu_T)+ \Phi(\mu_0) - \int_0^T e^{-t}   h(\mu_t) \dd t \\
& \leq  \Phi(\mu_0) + (1-e^{-T})\Phi(\mu_T) + (1-e^{-T}) \|h\|_\infty \\
& \leq  \Phi(\mu_0) + 2(1-e^{-T}) \|h\|_\infty\leq \Phi(\mu_0) + 2T \|h\|_\infty
\end{split}
\]
where in the last but one inequality  we use the trivial upper bound $\Phi(\mu_T)\leq \|h\|_\infty$.
\end{proof}

\subsection{A controlled EVI} \label{subsection:controlledEVI}

Let us recall the definition of EVI gradient flow and  collect all the properties we shall need in the sequel. Although the following discussion could be carried out in an abstract metric space $(\X,\sfd)$, we choose $(\X,\sfd) = (\cP_2(\R^d),W_2)$ as this will be our framework in the whole manuscript.

\begin{definition}
Given $\kappa \in \R$, a curve $(\mu_t)_{t \geq 0} \subseteq \cP_2(\R^d)$ is an $\EVI_\kappa$-gradient flow of $\cE$ provided it belongs to $AC_\loc((0,+\infty),\cP_2(\R^d)) \cap C([0,+\infty),\cP_2(\R^d))$ and
\[
\frac{\dd}{\dd t}W_2^2(\mu_t,\nu) + \frac{\kappa}{2}W_2^2(\mu_t,\nu) + \cE(\mu_t) \leq \cE(\nu), \qquad \forall \nu \in \cP_2(\R^d),\textrm{ a.e. } t>0.
\]
\end{definition}

From \cite[Theorem 3.5]{MuSa20} we know that the following properties hold:
\begin{enumerate}[(i)]
\item \emph{Contraction}.\ If $(\mu_t)$ is an $\EVI_\kappa$-gradient flow of $\cE$ starting from $\mu \in \overline{\cD(\cE)}$ and $(\nu_t)$ is a second $\EVI_\kappa$-gradient flow of $\cE$ starting from $\nu \in \overline{\cD(\cE)}$, then
\begin{equation}
\label{eq:contraction}
W_2(\mu_t,\nu_t) \leq e^{-\kappa t} W_2(\mu,\nu), \qquad \forall t \geq 0.
\end{equation}
This means that EVI-gradient flows are unique (provided they exist) and thus if there exists an EVI-gradient flow $(\mu_t)$ starting from $\mu$, then a 1-parameter semigroup $(\sfS_t)_{t \geq 0}$ is unambiguously associated to it via $\sfS_t(\mu) = \mu_t$.
\item \emph{Monotonicity}.\ For any $\mu \in \cP_2(\R^d)$, the map
\begin{equation}
\label{eq:monotonicity-entropy}
t \mapsto \cE(\sfS_t \mu) \quad \textrm{is non-increasing on } [0,\infty).
\end{equation}
\item \emph{Asymptotic expansion as $t \downarrow 0$}.\ If $\mu \in \cD(|\partial\cE|)$ and $\kappa \leq 0$, then for every $\nu \in \cD(\cE)$ and $t \geq 0$ it holds
\begin{equation}\label{eq:asymptotics}
\frac{e^{2\kappa t}}{2}W_2^2(\sfS_t\mu,\nu) - \frac12 W_2^2(\mu,\nu) \leq I_{2\kappa}(t)\left(\cE(\nu)-\cE(\mu)\right) + \frac{t^2}{2}|\partial\cE|^2(\mu),
\end{equation}
where $I_{2\kappa}(t) := \int_0^t e^{2\kappa s}\dd s$.
\end{enumerate}

As regards the existence, from \cite[Theorems 11.2.1 and 11.2.8]{AmGiSa08} we know that under Assumption \ref{ass: OT energy functional} $\cE$ introduced in \eqref{eq:def-functional2} generates  an $\EVI_\kappa$-gradient flow on $\cP_2(\R^d)$, where $\kappa :=\kappa_V+\kappa_W$.

\medskip

After this premise, let us now focus our attention on the admissible curves $\adm_T$, proving that they also satisfy an EVI-type inequality.

For sake of simplicity, we introduce the following notational convention for $(f^\dagger,g^\dagger) \in H_\dagger$ as in Definition \ref{definition:H}
\[
g^\dagger(\pi) = g^\dagger_\cE(\pi) + g^\dagger_{W_2}(\pi)
\]
where
\[
\begin{split}    
g^\dagger_{\cE}(\pi) & :=  a \Big(\cE(\rho) - \cE(\pi) - \frac{\kappa}{2} W_2^2(\pi,\rho) \Big) \\
& \qquad + \nabla\varphi\Big(\frac12 W_2^2(\pi,\bm\mu) \Big)\cdot\Big(\cE(\bm\mu) - \cE(\pi)\bm{1} - \frac{\kappa}{2} W_2^2(\pi,\bm\mu)\Big), \\
g^\dagger_{W_2}(\pi) & := \frac{1}{2} \left( a W_2(\pi,\rho) + \nabla\varphi\Big(\frac12 W_2^2(\pi,\bm\mu)\Big) \cdot W_2(\pi,\bm\mu) \right)^2.
\end{split}
\]
Likewise, for $(f^\ddagger,g^\ddagger) \in H_\ddagger$ as in Definition \ref{definition:H}, we write
\[
g^\ddagger(\mu) = g^\ddagger_\cE(\mu) + g^\ddagger_{W_2}(\mu)
\]
with
\[ 
\begin{split}
g^{\ddagger}_{\cE}(\mu) & := -a \Big(\cE(\gamma) - \cE(\mu) - \frac{\kappa}{2} W_2^2(\mu,\gamma)\Big) \\
& \qquad - \nabla\varphi\Big(\frac12 W_2^2(\mu,\bm\pi) \Big)\cdot\Big(\cE(\bm\pi) - \cE(\mu)\bm{1} - \frac{\kappa}{2} W_2^2(\mu,\bm\pi)\Big), \\
g^\ddagger_{W_2}(\mu) & := a^2 W_2^2(\mu,\gamma) - \frac{1}{2} \left(a W_2(\pi,\rho) + \nabla\varphi \Big(\frac12 W_2^2(\mu,\bm\pi)\Big) \cdot W_2(\mu,\bm\pi)\right)^2.
\end{split}
\]

Let us mention that $\cE$ may be negative, but only quadratically so. Indeed, we recall from \cite[Lemma 4.1]{CoKrTo21} that for any given $\nu \in \cP_2(\R^d)$, we can choose $c_1 \in( -\kappa, -\kappa + 1)$ and $c_2$ such that the functional
\begin{equation}\label{eq:quadratic-lower-bound}
\overline{\cE}(\mu) := \cE(\mu) + \frac{c_1}{2}W_2^2(\mu,\nu) + c_2
\end{equation}
has its infimum equal to $0$. Thus let us fix $\nu \in \cP_2(\bR^d)$ and $c_1,c_2$ as in \eqref{eq:quadratic-lower-bound} once for all.

We also introduce some notation about transportation maps that will be useful in the sequel.
We denote by 
$\otmap{\mu}{\nu}$ the optimal transport map that pushes $\mu$ onto $\nu$. We also denote by $\botmap{\mu}{\nu}=\otmap{\mu}{\nu}-\bm{id}$. For $\bm\mu=(\mu_1,\ldots,\mu_k)$ we write $\botmap{\pi}{\bm\mu}$ as short-hand notation for $(\botmap{\pi}{\mu_1},\ldots,\botmap{\pi}{\mu_k})$. Recall that, by \cite[Theorem 6.2.4]{AmGiSa08}, $\otmap{\mu}{\nu}$ exists whenever $\mu \ll \cL^d$.

The following Proposition contains a new version of EVI, that we name \emph{controlled EVI}. Due to the estimates that we can recover from it, this inequality  turns out  well suited to the study of controlled gradient flows. 

\begin{proposition}\label{prop:modif_EVI}
Let $(\pi_t,u_t) \in \admt$ and $f^\dagger, g^\dagger$ be as in Definition \ref{definition:H}. Then we have
\begin{equation}\label{eq:refined}
f^\dagger(\pi_T) - f^\dagger(\pi_0) \leq \int_0^T g^{\dagger}_{\cE}(\pi_t) - \ip{u_t}{a\botmap{\pi_t}{\rho} + \nabla\varphi \bigg(\frac{1}{2} W_2^2(\pi_t,\bm\mu) \bigg)\cdot\botmap{\pi_t}{\bm\mu}}_{L^2(\pi_t)} \dd t .
\end{equation}
Likewise, given $(\mu_t,u_t) \in \admt$ and $f^\ddagger, g^\ddagger$ be as in Definition \ref{definition:H}, it holds
\begin{equation}\label{eq:refined2}
f^\ddagger(\mu_T) - f^\ddagger(\mu_0) \geq \int_0^T g^{\ddagger}_{\cE}(\mu_t) + \ip{u_t}{a\botmap{\mu_t}{\gamma} + \nabla\varphi \bigg(\frac{1}{2} W_2^2(\mu_t,\bm\pi) \bigg)\cdot\botmap{\mu_t}{\bm\pi}}_{L^2(\mu_t)} \dd t .
\end{equation}


\end{proposition}

\begin{proof}
First of all, note that $\otmap{\pi_t}{\rho},\otmap{\pi_t}{\mu_i}$ exist for every $i=1,...,k$ and a.e.\ $t \in [0,T]$, because $\pi_t \ll \cL^d$ for a.e.\ $t \in [0,T]$ by definition of admissible curve. It is indeed implicitly understood that $\pi_t \ll \cL^d$ in \eqref{eq:controlled_GF}, where $P(\pi_t) = P(\frac{\dd\pi_t}{\dd\cL^d})$.

After this premise, observe that the $W_2$-absolute continuity of $\pi_t$ on $[0,T]$ (by definition of $\admt$) trivially implies that $t \mapsto W_2^2(\pi_t,\rho)$ and $t \mapsto W_2^2(\pi_t,\mu_i)$ are absolutely continuous on $[0,T]$ as well, for all $i=1,\dots,k$. Then, observe that the velocity field of $\pi_t$ is given by
\begin{equation}\label{eq:drift}
v_t = -\frac12\nabla U'(\pi_t) - \nabla V - \nabla W * \pi_t + u_t, \qquad \textrm{for a.e. } t \in [0,T],
\end{equation}
since by algebraic manipulations \eqref{eq:controlled_GF} can be rewritten as a continuity equation with the above $v_t$ as drift. Therefore, by \cite[Theorem 1.31]{gigli2012second}
\[
\frac{\dd}{\dd t}\left(\frac{1}{2}W_2^2(\pi_t,\mu_i)\right) = -\int \ip{v_t}{\botmap{\pi_t}{\mu_i}}\, \dd \pi_t = \int \ip{\frac{1}{2}\nabla U'(\pi_t) + \nabla V + \nabla W * \pi_t - u_t}{\botmap{\pi_t}{\mu_i}}\, \dd \pi_t
\]
for a.e.\ $t \in [0,T]$ and for all $i=1,...,k$, moreover the same apply for $\rho$ in place of $\mu_i$. Since $\varphi \in \cT$, we have that $t \mapsto \varphi(\frac{1}{2}W_2^2(\pi_t,\mu_i))$ and $t \mapsto \varphi(\frac{1}{2}W_2^2(\pi_t,\rho))$ are absolutely continuous too for all $i=1,\dots,k$, and by the chain rule  we have 
\begin{equation}\label{eq:applying-gigli}
\begin{split}
f^\dagger(& \pi_T) - f^\dagger(\pi_0) \\
& = \sum_{i=1}^k \int_0^T \partial_i\varphi \bigg(\frac{1}{2} W_2^2(\pi_t,\bm\mu) \bigg)\frac{\dd}{\dd t} \left(\frac{1}{2} W_2^2(\pi_t,\mu_i)\right)\dd t + a\int_0^T \frac{\dd}{\dd t} \left(\frac{1}{2} W_2^2(\pi_t,\rho)\right)\dd t \\
& = \sum_{i=1}^k \int_0^T \partial_i\varphi \bigg(\frac{1}{2} W_2^2(\pi_t,\bm\mu)\bigg)\int \ip{\frac{1}{2}\nabla U'(\pi_t) + \nabla V + \nabla W * \pi_t - u_t}{\botmap{\pi_t}{\mu_i}}\, \dd \pi_t\dd t \\
& \qquad + a\int_0^T \int \ip{\frac{1}{2}\nabla U'(\pi_t) + \nabla V + \nabla W * \pi_t - u_t}{\botmap{\pi_t}{\rho}}\, \dd \pi_t\dd t.
\end{split}
\end{equation}
Let us now observe that $\pi_t \in \cD(\cI)$ for a.e.\ $t \in [0,T]$. Indeed, the fact that $\pi_\cdot \in AC^2([0,T],\cP_2(\R^d))$ implies that $v_t \in L^2(\pi_t)$ for a.e.\ $t \in [0,T]$ by Remark \ref{rmk:continuity_equation}, and \eqref{eq:integrability-control} gives $u_t \in L^2(\pi_t)$ for a.e.\ $t \in [0,T]$. Therefore if we recall that $P(r) = rU'(r)-U(r)$ we deduce that
\[
\begin{split}
w_t := u_t - v_t & = \frac{1}{2}\nabla U'(\pi_t) + \nabla V + \nabla W * \pi_t \\
& = \frac{1}{2}\frac{\nabla P(\pi_t)}{\pi_t} + \nabla V + \nabla W * \pi_t \in L^2(\pi_t), \qquad \textrm{for a.e. } t \in [0,T]
\end{split}
\]
too. Since $L^2(\pi_t) \subset L^1(\pi_t)$, this yields
\[
\frac{1}{2}\nabla P(\pi_t) + \pi_t\nabla V + \pi_t(\nabla W * \pi_t) \in L^1(\R^d,\cL^d).
\]
Noting that $\nabla V$ and $\nabla W * \pi_t$ are locally bounded (since $V$ and $W$ are respectively $\kappa_V$- and $\kappa_W$-convex), we deduce that $\nabla P(\pi_t) \in L^1_\loc(\R^d)$. We are thus in position to apply \cite[Theorem 10.4.13]{AmGiSa08} (remarking that $P$ here coincides with $L_F$ there), which precisely grants that $\pi_t \in \cD(\cI)$ for a.e.\ $t \in [0,T]$.  Moreover, the same theorem ensures that $w_t$ belongs to the subdifferential of $\cE$ at $\pi_t$ (it is actually the element of minimal $L^2(\pi_t)$-norm). By \eqref{eq:subdifferential}, this means that
\begin{equation}\label{eq:W^2estimate}
\begin{split}
\frac{\dd}{\dd t}\left(\frac{1}{2}W_2^2(\pi_t,\rho)\right) & = \int \ip{\frac{1}{2}\nabla U'(\pi_t) + \nabla V + \nabla W * \pi_t}{\botmap{\pi_t}{\rho}}\, \dd \pi_t \\ &\leq \cE(\rho) - \cE(\pi_t) - \frac{\kappa}{2}W_2^2(\pi_t,\rho)
\end{split}
\end{equation}
and analogous estimates hold for $\mu_1,\dots\mu_k$ in place of $\rho$. Being $\partial_i \varphi >0$ for all $i=1\dots, k$, we can plug the above inequalities  into \eqref{eq:applying-gigli} providing us the desired conclusion.
\end{proof}

As a consequence of Proposition \ref{prop:modif_EVI}, we now obtain some quantitative estimates for controlled gradient flows that will be extensively used in the following.

\begin{lemma} \label{lemma:control_controlled_paths}
Let $(\mu_t,u_t) \in \admt$, for some $T>0$ fixed. Then for every $\alpha \in \R$ we have that for a.e.\ $t \in [0,T]$ it holds
\begin{equation}\label{eq:first-bound-derivative}
\frac{\dd}{\dd t} \left( \frac{1}{2} e^{\alpha t} W^2_2(\mu_t, \rho) \right) \leq e^{\alpha t} \left(\cE(\rho) - \cE(\mu_t) + \frac{(\alpha + 1 - \kappa)}{2} W_2^2(\mu_t,\rho) + \frac{1}{2}\|u_t\|^2_{L^2(\mu_t)} \right).
\end{equation}
In particular, if $\alpha \leq 3(\kappa - 1)$, there exists a non-negative constant
\begin{equation} \label{eqn:definitionMrhonu}
    M_{\rho,\nu} := \max\left\{\cE(\rho) + c_1 W_2^2(\rho,\nu) + c_2,0\right\}
\end{equation}
such that:
\begin{enumerate}[(a)]
\item \label{item:lemma:control_controlled_paths_differentialcontrol} The a.e.\ derivative of $t \mapsto \frac{1}{2} e^{\alpha t} W^2_2(\mu_t, \rho)$ can be controlled as:
\[
\frac{\dd}{\dd t} \left( \frac{1}{2} e^{\alpha t} W^2_2(\mu_t, \rho) \right) \leq e^{\alpha t} \left(M_{\rho,\nu} + \frac{1}{2}\vn{u_t}^2_{L^2(\mu_t)} \right).
\]
\item \label{item:lemma:control_controlled_paths_metriccontrol} The growth of the metric can be controlled, for all $t \in [0,T]$, as:
\[
\frac{1}{2} e^{\alpha t} W_2^2(\mu_t,\rho) \leq \frac{1}{2} W_2^2(\mu_0,\rho) + M_{\rho,\nu}\frac{e^{\alpha t}-1}{\alpha} + \frac12 \int_0^t e^{\alpha s}\vn{u_s}^2_{L^2(\mu_s)} \dd s.
\]
In particular,
\begin{equation}\label{eq:uniform-control-distance}
\begin{split}
\frac12 e^{-\alpha^- T}\sup_{t \in [0,T]} W_2^2(\mu_t,\rho) & \leq \frac12 W_2^2(\mu_0,\rho) + M_{\rho,\nu}\frac{e^{\alpha T}-1}{\alpha} \\
& \qquad + \frac12 e^{\alpha^+ T}\int_0^T \vn{u_s}^2_{L^2(\mu_t)} \dd t,
\end{split}
\end{equation}
where $\alpha^\pm := \max\{\pm\alpha,0\}$ and $\alpha^{-1}(e^{\alpha T}-1)$ has to be understood as $T$ if $\alpha=0$.
\end{enumerate}
\end{lemma}

\begin{proof}
As $(\mu_t) \in AC^2([0,T],\cP_2(\R^d))$, the map $t \mapsto W^2_2(\mu_t,\rho)$ is absolutely continuous and a fortiori so is $t \mapsto \frac{1}{2} e^{\alpha t} W^2_2(\mu_t,\rho)$. This ensures the a.e.\ existence of its derivative. In order to bound it from above, using \eqref{eq:W^2estimate} in Proposition \ref{prop:modif_EVI}, applying Young's inequality to estimate $\ip{u_t}{\botmap{\mu_t}{\rho}}_{L^2(\mu_t)}$ and observing that $\|\botmap{\mu_t}{\rho}\|_{L^2(\mu_t)} = W_2(\mu_t,\rho)$ we find
\[
\begin{split}
\frac{\dd}{\dd t} & \left( \frac{1}{2} e^{\alpha t} W^2_2(\mu_t, \rho) \right) \\
& \leq e^{\alpha t} \left(\cE(\rho) - \cE(\mu_t) + \frac{(\alpha + 1 - \kappa)}{2} W^2(\mu_t,\rho) + \frac{1}{2}\vn{u_t}^2_{L^2(\mu_t)} \right) \\
& =e^{\alpha t} \left(\cE(\rho) - \overline{\cE}(\mu_t) + \frac{c_1}{2} W_2^2(\mu_t,\nu) + c_2 + \frac{(\alpha + 1 - \kappa)}{2} W^2_2(\mu_t,\rho) + \frac{1}{2}\vn{u_t}^2_{L^2(\mu_t)} \right),
\end{split}
\]
where on the second line we used \eqref{eq:quadratic-lower-bound}. Note also that the first inequality provides us with \eqref{eq:first-bound-derivative}. To proceed further, using $\overline{\cE} \geq 0$ and the triangular inequality combined with the elementary estimate $\frac{1}{2}(a+b)^2 \leq a^2 + b^2$ yields
\begin{multline*}
\frac{\dd}{\dd t} \left( \frac{1}{2} e^{\alpha t} W^2_2(\mu_t, \rho) \right) \\
\leq e^{\alpha t} \left(\cE(\rho) + c_1 W_2^2(\rho,\nu) + c_2 + \frac{\alpha + 1 + 2c_1 - \kappa}{2} W^2_2(\mu_t,\rho) + \frac{1}{2}\vn{u_t}^2_{L^2(\mu_t)} \right).
\end{multline*}
Now, recalling that $c_1 \in (-\kappa,-\kappa + 1)$ and assuming $\alpha \leq 3(\kappa - 1)$, so that $\alpha + 1 + 2c_1 - \kappa \leq 2(\kappa-1) + 2c_1 < 0$, the above inequality finally rewrites as
\[
\frac{\dd}{\dd t} \left( \frac{1}{2} e^{\alpha t} W^2_2(\mu_t, \rho) \right) \leq e^{\alpha t} \left(\cE(\rho) + c_1 W_2^2(\rho,\nu) + c_2 + \frac{1}{2}\vn{u_t}^2_{L^2(\mu_t)} \right).
\]
Claim \ref{item:lemma:control_controlled_paths_differentialcontrol} then follows by using the definition of $M_{\rho,\nu}$ of \eqref{eqn:definitionMrhonu}, while for claim \ref{item:lemma:control_controlled_paths_metriccontrol}  it is sufficient to integrate \ref{item:lemma:control_controlled_paths_differentialcontrol} on $[0,T]$.
\end{proof}

\subsection{Continuity of the transport map} \label{subsection:continuityTransportMap}
In this section we state a result on the continuity of the transport map that will be used for the proof of the supersolution property.

\begin{lemma}\label{lem:continuity_map}
Let $(\mu^n)_{n \in \bN} \subset \cP_2(\R^d)$ be weakly converging to some $\mu^0 \in \cP_2(\R^d)$ and let $\gamma \in \cP_2(\R^d)$ be such that $\gamma \ll \cL^d$. Assume moreover that $W_2(\mu^n,\gamma) \to W_2(\mu^0,\gamma)$ as $n \to \infty$. Then $\botmap{\gamma}{\mu^n}$ converges strongly to $\botmap{\gamma}{\mu^0}$ in $L^2(\gamma)$.
\end{lemma}



\begin{proof}
From \cite[Corollary 5.23]{Vi09} we have that $\botmap{\gamma}{\mu^n}$ converges to $\botmap{\gamma}{\mu^0}$ in probability, i.e.\ for all $\eps>0$ we have 
\[
\lim_{n \to \infty}\gamma\big( |\botmap{\gamma}{\mu^n} - \botmap{\gamma}{\mu^0}|\geq \varepsilon\big) = 0.
\]
Moreover, recalling that $W_2(\mu^n,\gamma) = \|\botmap{\gamma}{\mu^n}\|_{L^2(\gamma)}$ and that $W_2(\mu^n,\gamma) \to W_2(\mu^0,\gamma)$, we have 
\begin{equation}\label{eq:conv_of_maps_1}
\lim_{n \to \infty} \|\botmap{\gamma}{\mu^n}\|_{L^2(\gamma)} = \|\botmap{\gamma}{\mu^0}\|_{L^2(\gamma)}.
\end{equation}
The conclusion follows from the fact that convergence in probability plus convergence of the norms implies convergence in the strong $L^2$-sense. Indeed, because of \eqref{eq:conv_of_maps_1} we know that $(\botmap{\gamma}{\mu^n})_{n \in \bN}$ is a bounded sequence in $L^2(\gamma)$ and therefore, up to extracting a subsequence, it converges weakly in $L^2(\gamma)$ to some $\tilde{\bm{t}}$. We proceed to show that $\tilde{\bm{t}} = \botmap{\gamma}{\mu^0}$, so that the choice of the subsequence plays no role. To this end, take any $v \in L^2 \cap L^\infty(\R^d)$ and, using the weak $L^2(\gamma)$-convergence, note that
\[
\lim_{n \to \infty} \int \ip{\botmap{\gamma}{\mu^n}}{v}\, \dd\gamma = \int \ip{\tilde{\bm{t}}}{v}\, \dd\gamma.
\]
On the other hand, fix $\eps>0$. We have
\[
\begin{split}
\Big|\int \ip{\botmap{\gamma}{\mu^n} & - \botmap{\gamma}{\mu^0}}{v} \dd\gamma\Big| \leq \|v\|_{\infty} \int |\botmap{\gamma}{\mu^n} - \botmap{\gamma}{\mu^0}| \mathbf{1}_{\{ |\botmap{\gamma}{\mu^n} - \botmap{\gamma}{\mu^0}|\geq \eps \}} \dd\gamma + \eps \|v\|_{\infty}\\
& \leq \|v\|_{\infty} \sqrt{\int |\botmap{\gamma}{\mu^n} - \botmap{\gamma}{\mu^0}|^2\, \dd\gamma}\,\sqrt{\gamma\big( |\botmap{\gamma}{\mu^n} - \botmap{\gamma}{\mu^0}| \geq \eps\big) } + \eps \|v\|_{\infty}\\
& \leq \|v\|_{\infty} \sqrt{2\int \Big( |\botmap{\gamma}{\mu^n}|^2 + |\botmap{\gamma}{\mu^0}|^2\Big) \dd\gamma}\,\sqrt{\gamma\big( |\botmap{\gamma}{\mu^n} - \botmap{\gamma}{\mu^0}|\geq \eps\big) } + \eps \|v\|_{\infty}\\
& \leq \|v\|_{\infty}\sqrt{2}\left( \sup_{n \in \mathbb N}W_2(\mu^n,\gamma) + W_2(\mu^0,\gamma) \right) \sqrt{\gamma\big( |\botmap{\gamma}{\mu^n} - \botmap{\gamma}{\mu^0}|\geq \eps\big)} + \eps \|v\|_{\infty}.
\end{split}
\]
Using convergence in probability and the fact that $W_2(\mu^n,\gamma) \to W_2(\mu^0,\gamma)$ as $n \to \infty$ we obtain from the above
\[
\limsup_{n \to \infty} \Big|\int \ip{\botmap{\gamma}{\mu^n} - \botmap{\gamma}{\mu^0}}{v} \,\dd\gamma\Big| \leq \eps 
\|v\|_{\infty}.
\]
Since the choice of $\eps$ is arbitrary we deduce that 
\[
\lim_{n \to \infty} \int \ip{\botmap{\gamma}{\mu^n}}{v} \,\dd\gamma = \int \ip{\botmap{\gamma}{\mu^0}}{v} \,\dd\gamma
\]
and since the choice of $v$ was arbitrary too, by a standard density argument we obtain $\tilde{\bm{t}} = \botmap{\gamma}{\mu^0}$. Therefore we can conclude that $\botmap{\gamma}{\mu^n}$ converges weakly in $L^2(\gamma)$ to $\botmap{\gamma}{\mu^0}$. Since the $L^2(\gamma)$ norms also converge, see \eqref{eq:conv_of_maps_1}, and $L^2(\gamma)$ is a Hilbert space, this yields strong convergence in $L^2(\gamma)$.
\end{proof}

\section{The subsolution property}\label{sec:subsolution} 

In this section we work towards the proof that $\Phi^*$, the upper semicontinuous relaxation of $\Phi$ w.r.t.\ the $W_p$-topology for $p<2$, is a viscosity subsolution for 
\begin{equation}\label{eqn:HJ_subsol_strategy}
f - \lambda H_\dagger f = h
\end{equation}
for a bounded $W_p$-continuous function $h$. As before, without loss of generality, the proof will be carried out for the case $\lambda = 1$.

To indicate the strategy of the proof, note that the definition of viscosity subsolutions in Definition \ref{definition:viscosity_solutions_HJ_sequences} can be simplified in the context where optimizers exist: $\Phi^*$ is a subsolution for \eqref{eqn:HJ_subsol_strategy} if for any $(f^\dagger,g^\dagger) \in H_\dagger$ there exists an optimizer $\pi^0$ such that
\begin{gather}
\Phi^*(\pi^0) - f^\dagger(\pi^0) = \sup\{\Phi^* - f^\dagger\},  \label{eqn:visc_subsol_strategy1} \\
\Phi^*(\pi^0) - g^\dagger(\pi^0) \leq h(\pi^0). \label{eqn:visc_subsol_strategy2}
\end{gather}
In the proof below, we will start with a sequence of optimizers $\pi_0^n$ for
\begin{equation} \label{eqn:visc_subsol_strategy3}
\lim_{n \to \infty} \Phi(\pi^n_0) - f^\dagger(\pi^n_0) = \sup\{\Phi - f^\dagger\}
\end{equation}
and we aim to show that this sequence has a limit $\pi^0$ for which \eqref{eqn:visc_subsol_strategy1} and \eqref{eqn:visc_subsol_strategy2} hold. To establish the latter, we employ the classical strategy of following curves that are close to the optimal control for an infinitesimal amount of time.

As our test functions include the energy functional $\cE$, this latter step forces us to control the behaviour of $\cE$ along our controlled curves. To do so, we replace the $\pi^n_0$ by $\pi^n$ that are obtained from the almost optimally controlled curves started from our original points $\pi^n_0$. As the controlled curves to large extent follow the gradient flow, this gives a control on $\cE$ in the chosen $\pi^n$.

\smallskip

In the following lemma, we start out by analyzing an almost optimally controlled curve $(\pi_{\cdot},u_{\cdot})$ started from an almost optimizer of $\Phi - f^\dagger$. We control the cost of the control on small time intervals and find a small time $t_0$ for which we control also $\cE(\pi_{t_0})$.

Afterwards, in Proposition \ref{proposition:existence_regularity_minimizer_subsol}, we show that starting from measures ever closer optimizing  $\Phi - f^\dagger$ with an almost optimally controlled curve, we can find a sequence $\pi^n$ that can be shown to have a weak limit point $\pi^0$. Finally in the proof of Proposition \ref{prop:subsolution} we show that $\pi^0$ satisfies \eqref{eqn:visc_subsol_strategy2}.

\begin{lemma} \label{lemma:raw_estimate_Phifdagger}
Fix $0<\eps<1$ and a time horizon $T<T_0$, with $T_0$ small enough. Let 
\[
f^{\dagger}(\pi) := \frac{a}{2} W_2^2(\pi,\rho) +\varphi\left(\frac{1}{2}W^2_2(\pi,\bm\mu)\right)
\]
be an admissible test function and $(\pi_{\cdot},u_{\cdot})$ be an admissible curve such that 
\begin{align}
& \Phi(\pi_0) - f^\dagger(\pi_0) \geq \sup \{\Phi - f^\dagger\} - \eps \label{eqn:raw_estimate_Phifdagger_optimizer}  \\
& \Phi(\pi_0) \leq \cA_T(\pi_{\cdot},u_{\cdot})+ e^{-T} \Phi(\pi_T) + \eps \label{eqn:raw_estimate_Phifdagger_almostDPP} 
\end{align}
where $\cA_T$ is defined at \eqref{eq:action_T}.

Then there exist:
\begin{enumerate}[(i)]
\item $t_0 \in [0,T]$,
\item a constant {$R> 0$} depending on $T_0$, $\kappa$, $\vn{h}_{\infty}$, $\cE(\rho)$, $W_2(\rho,\bm\mu)$, $W_2(\rho,\nu)$, $c_1$, and $c_2$ (where $\nu$, $c_1$, and $c_2$ have been fixed in \eqref{eq:quadratic-lower-bound}),
\item a constant {$M\geq 0$} depending on the same parameters of $R$ and also on {$a, \cE(\bm\mu)$ and } 
\[
S := \sup_{[0,+\infty)^k \cap B_{R'}}|\nabla\varphi|, \qquad \textrm{where } R' := R(1+W_2^2(\pi_0,\rho)),
\]
\end{enumerate}
such that
\[
a\overline{\cE}(\pi_{t_0}) + \frac{1}{2T} \int_0^T\Big(e^{-t} - \frac12\Big)  \vn{u_t}^2_{L^2(\pi_t)} \dd t \leq M \Big(1+W_2^2(\pi_0,\rho) + \frac{\eps}{T}\Big).
\]

\end{lemma}

\begin{proof}
As a first step, we prove that the curve $t \mapsto \pi_t$ stays within a ball centered at $\rho$ whose radius can be controlled only in terms of structural constants (i.e.\ those listed at point (ii) of the statement) and on $W_2(\pi_0,\rho)$. To this end, we start observing that by \eqref{eqn:raw_estimate_Phifdagger_almostDPP}
\[
\Phi(\pi_0) \leq -\frac12 \int_0^T e^{-t} \|u_t\|^2_{L^2(\pi_t)} \dd t + e^{-T} \|h\|_\infty + e^{-T}\Phi(\pi_T) + \eps,
\]
so that, if we recall that $|\Phi| \leq \|h\|_\infty$ by Corollary \ref{corollary:Phi_usc_admissiblecurves}, the previous inequality turns into
\[
\frac12 e^{-T} \int_0^T \|u_t\|^2_{L^2(\pi_t)} \dd t \leq \frac12 \int_0^T e^{-t} \|u_t\|^2_{L^2(\pi_t)} \dd t \leq 3\|h\|_\infty + \eps.
\]
Plugging this information into \eqref{eq:uniform-control-distance} yields
\[
\frac12 e^{-\alpha^- T}\sup_{t \in [0,T]} W_2^2(\pi_t,\rho) \leq \frac12 W_2^2(\pi_0,\rho) + M_{\rho,\nu}\frac{e^{\alpha T}-1}{\alpha} + e^{(\alpha^+ + 1) T}(3\|h\|_\infty + \eps),
\]
leading to

\begin{equation}\label{eq:serotonine}
\frac12 \sup_{t \in [0,T]} W_2^2(\pi_t,\rho) \leq R(1+W_2^2(\pi_0,\rho)).
\end{equation}

For later purpose, let us note that by triangle inequality (and up to choosing a larger $R$ which incorporates $W_2(\rho,\mu_i)$) we also have
\begin{equation}\label{eq:serotonine2}
\frac12 \sup_{t \in [0,T]} W_2^2(\pi_t,\mu_i) \leq R(1+W_2^2(\pi_0,\rho)), \qquad \forall i=1,\dots,k.
\end{equation}
After this premise, let us stress that in what follows, $M$ denotes a non-negative constant depending on $T_0$, $\kappa$, $\vn{h}_{\infty}$, $\cE(\rho)$, $\cE(\bm\mu)$, $W_2(\rho,\bm\mu)$, $W_2(\rho,\nu)$, $S$ only, whose numerical value may change from line to line. 

On the one hand, let us start observing that 
\begin{equation}\label{eq:raw_1}
\begin{split}
\frac{1}{T} \big(f^{\dagger}(\pi_T)-f^{\dagger}(\pi_0)\big) & \stackrel{\eqref{eqn:raw_estimate_Phifdagger_optimizer}}{\geq} \frac{1}{T}\big(\Phi(\pi_T)-\Phi(\pi_0)-\eps\big) \\
&\stackrel{\eqref{eqn:raw_estimate_Phifdagger_almostDPP}}{\geq} -\frac{1}{T}\cA_T(\pi_{\cdot},u_{\cdot}) + \frac{1}{T}(1-e^{-T})\Phi(\pi_T)-\frac{2\eps}{T}\\
&\geq -M +\frac{1}{2T}\int_{0}^{T}e^{-t}\vn{u_t}^2_{L^2(\mu_t)} \dd t-\frac{2\eps}{T}.
\end{split}
\end{equation}

On the other hand, applying Cauchy-Schwarz inequality to Proposition \ref{prop:modif_EVI} we deduce that 
\[
\begin{split}
\frac{1}{T} \big(f^{\dagger}(\pi_T)-f^{\dagger}(\pi_0)\big) & \leq \frac{1}{T}\int_{0}^{T} g_\cE^\dagger(\pi_t)\,\dd t + \frac{1}{T} \int_{0}^T a W_2(\pi_t,\rho) \vn{u_t}_{L^2(\mu_t)}\dd t \\
& \qquad + \frac{1}{T}\int_0^T |\nabla\varphi\left(\frac12 W_2^2(\pi_t,\bm\mu)\right)| \cdot W_2(\pi_t,\bm\mu) \vn{u_t}_{L^2(\mu_t)}\dd t,
\end{split}
\]
where $|\nabla\varphi(\frac12 W_2^2(\pi_t,\bm\mu))| \cdot W_2(\pi_t,\bm\mu) := \sum_{i=1}^k |\partial_i\varphi(\frac12 W_2^2(\pi_t,\bm\mu))| W_2(\pi_t,\mu_i)$. An application of Young's inequality and of the triangle inequality together with \eqref{eq:serotonine} and \eqref{eq:serotonine2} then allows to bound the second and third terms on the right-hand side as 
\begin{equation}\label{eq:easy_term}
\begin{split}
\frac{1}{T} \int_{0}^T \Big(a W_2(\pi_t,\rho) + |\nabla \varphi\left(\frac12 W_2^2(\pi_t,\bm\mu)\right)|\cdot W_2(\pi_t,\bm\mu)\Big)\vn{u_t}_{L^2(\pi_t)}\dd t\\
\leq M\Big(1+\frac{1}{T} \int_{0}^T W^2_2(\pi_t,\rho)\dd t\Big) + \frac{1}{{{4}}T}\int_0^T\vn{u_t}^2_{L^2(\pi_t)} \dd t.
\end{split}
\end{equation}
For the first term, we argue as follows: first of all, by \eqref{eq:quadratic-lower-bound} we have
\[
-\cE(\pi_t) \leq -\overline\cE(\pi_t) + M\big(1+W_2^2(\pi_t,\rho)\big)\leq M\big(1+W_2^2(\pi_t,\rho)\big)
\]
where in the second inequality we  used the fact that $\overline\cE \geq 0$, and this inequality, together with \eqref{eq:serotonine2} and the fact that $\partial_i\varphi \geq 0$ for all $i=1,...,k$, implies
\[
\begin{split}
-\partial_i\varphi\bigg(\frac12 W_2^2(\pi_t,\bm\mu) \bigg) \cE(\pi_t) & \leq \partial_i\varphi\bigg(\frac12 W_2^2(\pi_t,\bm\mu) \bigg) \Big(  M\big(1 + W_2^2(\pi_t,\rho)\big)\Big) \\
& \leq M\big(1 + W_2^2(\pi_t,\rho)\big).
\end{split}
\]
 Plugging this bound and the previous one into the very definition of $g^\dagger_\cE$ (together with the triangle inequality {or estimate \eqref{eq:serotonine2}}) gives
\[
\frac{1}{T}\int_0^T g^\dagger_\cE(\pi_t)\,\dd t \leq -\frac{a}{T}\int_0^T \overline\cE(\pi_t)\,\dd t + M\Big(1+\frac{1}{T}\int_0^TW_2^{2}(\pi_t,\rho)\,\dd t\Big).
\]

Let now $t_0$ be a minimizer for $\overline\cE(\pi_t)$ over the time interval $[0,T]$ (we know that such a minimizer exists since $t \mapsto \overline\cE(\pi_t)$ is lower semicontinuous) and observe that this yields
\[
\frac{1}{T}\int_0^T g_\cE^\dagger(\pi_t)\,\dd t \leq -a\overline\cE(\pi_{t_0}) + M\Big(1+\frac{1}{T}\int_0^TW_2^{2}(\pi_t,\rho)\,\dd t\Big),
\]
so that by combining this estimate with \eqref{eq:easy_term} we establish 
\[
\begin{split}
\frac{1}{T} \big(f^{\dagger}(\pi_T)-f^{\dagger}(\pi_0)\big) & \leq -a\overline\cE(\pi_{t_{0}}) +  M\Big(1+\frac{1}{T} \int_{0}^T W^2_2(\pi_t,\rho)\,\dd t\Big) + \\
& \qquad \frac{1}{4T}\int_0^T \vn{u_t}^2_{L^2(\pi_t)} \dd t.
\end{split}
\]
If we plug this bound into \eqref{eq:raw_1} and rearrange the terms, we get
\[
a\overline{\cE}(\pi_{t_{0}}) + \frac{1}{2T} \int_0^T \Big(e^{-t} - \frac12\Big)  \vn{u_t}^2_{L^2(\mu_t)} \dd t \leq M \Big(1 + \frac{1}{T}\int_{0}^T W^2_2(\pi_t,\rho)\,\dd t +\frac{\eps}{T}\Big)
\]
and it is now sufficient to use \eqref{eq:serotonine} to obtain the desired result.
\end{proof}

In view of Proposition \ref{proposition:existence_regularity_minimizer_subsol}, let us also briefly recall for reader's sake some compactness and lower semicontinuity properties in the Wasserstein space.

\begin{lemma}\label{lemma:Wp-compactness}
Let $(\mu_n)_{n \in \bN}$ be a bounded sequence in $(\cP_2(\R^d),W_2)$. Then it is relatively compact in $(\cP_p(\R^d),W_p)$, for any $1 \leq p < 2$, and
\[
W_2(\mu_\infty,\nu) \leq \liminf_{k \to \infty} W_2(\mu_{n_k},\nu)
\]
for any $W_p$-convergent subsequence $(\mu_{n_k})_{k \in \bN}$ (with limit point $\mu_\infty$) and for any $\nu \in \cP_2(\R^d)$.
\end{lemma}

\begin{proof}
First of all, $(\mu_n)_{n \in \bN} \subseteq \cP_2(\R^d)$ is uniformly tight, since the second moments are uniformly bounded by $W_2$-boundedness, and balls in $\R^d$ are relatively compact, so that we can apply \cite[Remark 5.1.5]{AmGiSa08}. As a consequence, $(\mu_n)_{n \in \bN}$ is relatively compact w.r.t.\ narrow convergence, so that up to extracting a subsequence we may assume that $(\mu_n)_{n \in \bN}$ is indeed narrowly convergent to some limit $\mu_\infty$.

Then, by tightness and H\"older inequality, we can deduce that the $p$-th moment of $\mu_n$ w.r.t.\ any fixed reference point, $p<2$, converge to the $p$-th moment of $\mu_\infty$. Hence, $W_p(\mu_n,\mu_\infty) \to 0$ as $n \to \infty$.

The final part of the statement follows from the fact that $W_2$ is lower semicontinuous w.r.t.\ narrow convergence (see \cite[Proposition 3.5]{AmbrosioGigli11}, thus w.r.t.\ $W_p$-convergence as well.
\end{proof}

The next proposition is pivotal in the proof of the subsolution property for $\Phi^*$, as there we build a sequence $(\pi^n)_{n \geq 1}$ of almost optimizer for $\Phi - f^\dagger$ with quantified controls on distance and energy such that for any sequence $(\tilde{\pi}_\cdot^n,\tilde{u}_\cdot^n)$ of almost optimally controlled curves starting from $\pi^n$ we are able to uniformly bound the time-average of the $L^2$-norm of the controls $\tilde{u}^n_\cdot$.

\begin{proposition} \label{proposition:existence_regularity_minimizer_subsol}
Let
\[
f^\dagger(\pi) := \frac{a}{2} W_2^2(\pi,\rho) + \varphi\left(\frac12 W^2_2(\pi,\bm\mu)\right)
\]
be an admissible test function and fix $1 \leq p < 2$. Then there exists a sequence $(\pi^n)_{n \geq 1}$ converging to some $\pi^0$ w.r.t.\ $W_p$ such that:
\begin{enumerate}[(a)]
\item \label{item:subsol_optimizers_convergence_sup} The sequence $(\pi^n)_{n \geq 1}$ is optimizing for $\sup \{\Phi - f^\dagger\}$:
\[
\Phi(\pi^n) - f^\dagger(\pi^n) \geq \sup \left\{\Phi - f^\dagger \right\} { + O(T_n)}
\]
with $T_n=n^{-1}$.
\item \label{item:subsol_optimizers_regularization} The limit point $\pi^0$ is optimal for $\sup \{\Phi - f^\dagger\}$ and $\sup \{\Phi^* - f^\dagger\}$, where $\Phi^*$ denotes the upper semicontinuous relaxation of $\Phi$ w.r.t.\ the $W_p$ topology:
\[
\Phi^*(\pi^0) -f^\dagger(\pi^0) = \sup\{\Phi^* - f^\dagger\} = \sup\{\Phi - f^\dagger\}.
\]
\item \label{item:subsol_optimizers_convergence_metric} We have convergence of the metric:
\[
\lim_{n \to \infty} W_2(\pi^n,\rho) = W_2(\pi^0,\rho) \qquad \textrm{and} \qquad \lim_{n \to \infty} W_2(\pi^n,\mu_i) = W_2(\pi^0,\mu_i),
\]
for all $i=1,\ldots,k$.
\item \label{item:subsol_optimizers_convergence_valuefunction} We have convergence of the value function:
\begin{equation}\label{eq:almost-optimality}
\lim_{n \to \infty} \Phi(\pi^n) = \Phi^*(\pi^0).
\end{equation}
\item \label{item:subsol_optimizers_convergence_time} Let $\tau_n = n^{-1/2}$ and $(\tilde\pi^n_t, \tilde u^n_t)$ be an admissible curve starting from $\pi^n$ which is $T_n$-optimal in \eqref{eq:DPP} on the time interval $[0,\tau_n]$, that is
\begin{equation}\label{eq:toxicity}
\Phi(\pi^n) \leq \cA_{\tau_n}(\tilde\pi^n_\cdot, \tilde u^n_\cdot) + e^{-\tau_n} \Phi(\tilde \pi^n_{\tau_n}) + T_n.
\end{equation}
Then
\begin{equation}\label{eq:reg_subs_5}
\sup_{n \geq 1} \frac{1}{\tau_n} \int_0^{\tau_n} \vn{\tilde u^n_t}^2_{L^2(\tilde{\pi}^n_t)} \dd t < +\infty.
\end{equation}
Furthermore, let $0\leq (t_n)_{n \geq 1}$ be a sequence such that $t_n\leq \tau_n$ for all $n$. Then $\tilde\pi^n_{t_n}$ converges to $\pi^0$ w.r.t.\ $W_p$ and
\begin{equation}\label{eq:patience}
\lim_{n \to \infty} W_2(\tilde\pi^n_{t_n},\rho) = W_2(\pi^0,\rho), \qquad \lim_{n \to \infty} W_2(\tilde\pi^n_{t_n},\mu_i) = W_2(\pi^0,\mu_i),
\end{equation}
for all $i=1,\ldots,k$.
\end{enumerate}
\end{proposition}

\begin{proof}
\underline{Preliminary facts.} For any $n>0$, let $(\pi^n_t,u^n_t)_{t\geq0}$ satisfy \eqref{eqn:raw_estimate_Phifdagger_optimizer} and \eqref{eqn:raw_estimate_Phifdagger_almostDPP} for the choice $\eps = T = T_n := 1/n$. That is to say,
\begin{align}
& \Phi(\pi^n_0) - f^\dagger(\pi^n_0) \geq \sup\{\Phi - f^\dagger\} - T_n, \label{eqn:raw_estimate_Phifdagger_optimizer_n} \\
& \Phi(\pi^n_0) \leq \cA_{T_n}(\pi^n_{\cdot},u^n_{\cdot}) + e^{-T_n} \Phi(\pi^n_{T_n}) + T_n. \label{eqn:raw_estimate_Phifdagger_almostDPP_n}
\end{align}
By Corollary \ref{corollary:Phi_usc_admissiblecurves} $|\Phi| \leq \|h\|_\infty$, so that $(\Phi(\pi_0^n))_{n \geq 1}$ is a bounded sequence; by \eqref{eqn:raw_estimate_Phifdagger_optimizer_n} this implies that $(f^{\dagger}(\pi^n_0))_{n \geq 1}$ is an upper bounded sequence. From this and the fact that $\varphi$ is lower bounded (by continuity and the fact that $\partial_i\varphi >0 0$ for all $i$), we deduce the former of the following
\begin{equation}\label{eq:reg_subs_2}
\sup_{n \geq 1} W_2(\pi^n_0,\rho) < +\infty, \qquad \sup_{n \geq 1} \int_0^{T_0} \vn{u^n_t}^2_{L^2(\mu^n_t)}\dd t < +\infty,
\end{equation} 
while the latter follows from Lemma \ref{lemma:raw_estimate_Phifdagger} together with the former. Moreover, applying the same lemma with $T=T_n$ we find that there exists a sequence $t_n \in [0,T_n]$ such that
\begin{equation}\label{eq:reg_subs_1}
\sup_{n \geq 1} \bigg\{\overline\cE(\pi^n_{t_n}) + \frac{1}{T_n}\int_0^{T_n} \vn{u^n_t}^2_{L^2(\mu^n_t)} dt\bigg\} < +\infty.
\end{equation}
Consider now the sequence $\pi^n := \pi^n_{t_n}$. From Lemma \ref{lemma:control_controlled_paths}\ref{item:lemma:control_controlled_paths_metriccontrol} together with \eqref{eq:reg_subs_2} we obtain that $(\pi^n)_{n \geq 1}$ is bounded in $(\cP_2(\R^d),W_2)$, hence relatively compact in $(\cP_p(\R^d),W_p)$ by Lemma \ref{lemma:Wp-compactness}. We can thus extract a (non-relabeled) $W_p$-converging subsequence whose limit we shall denote $\pi^0$. We now proceed to show that $(\pi^n)_{n\geq1}$ and $\pi^0$ have the desired properties.

\medskip

\noindent\underline{Proof of \ref{item:subsol_optimizers_convergence_sup}.} First of all, applying Lemma \ref{lemma:control_controlled_paths}\ref{item:lemma:control_controlled_paths_metriccontrol} to the curve $(\pi^n_{\cdot},u^n_{\cdot})$ with the choices $T=T_n$, $\rho\in\{\rho,\mu_i\}$ and using \eqref{eq:reg_subs_1} give
\[
W_2^2(\pi^n,\rho)-W^2_2(\pi^n_0,\rho) \leq O(T_n) \qquad \textrm{and} \qquad W_2^2(\pi^n,\mu_i) - W^2_2(\pi^n_0,\mu_i) \leq O(T_n) 
\]
for all $i=1,\ldots,k$, which implies
\begin{equation}\label{eq:reg_subs_3}
f^\dagger(\pi^n) - f^\dagger(\pi^n_0) \leq O(T_n).
\end{equation}
Next, observe that 
\begin{equation}\label{eq:reg_subs_4}
\begin{split}
\Phi(\pi^n_{t_n}) & \stackrel{\eqref{eq:DPP}}{\geq} \cA_{t_n,T_n}(\pi^n_{\cdot},u^n_{\cdot}) + e^{-(T_n-t_n)}\Phi(\pi^n_{T_n}) \\
&\stackrel{\eqref{eq:reg_subs_1}}{\geq} e^{-(T_n-t_n)}\Phi(\pi^n_{T_n}) - O(T_n) \\
&\stackrel{\eqref{eqn:raw_estimate_Phifdagger_almostDPP_n}}{\geq} e^{t_n}\big(\Phi(\pi^n_{0}) - \cA_{T_n}(\pi^n_\cdot,u^n_\cdot)\big) - O(T_n) \\
&\geq \Phi(\pi^n_{0}) - O(T_n)\,,
\end{split}   
\end{equation}
where $\cA_{t_n,T_n}$ is defined as in \eqref{eq:action_T} and in the last inequality we used $e^{t_n}=1 + O (T_n)$. Thus, gathering \eqref{eq:reg_subs_3} and \eqref{eq:reg_subs_4} proves \ref{item:subsol_optimizers_convergence_sup}.

\medskip

\noindent\underline{Proof of \ref{item:subsol_optimizers_regularization}.} Note that $-f^\dagger$ is upper semicontinuous w.r.t.\ $W_p$ convergence: this follows from the $W_p$-lower semicontinuity of $W_2(\cdot,\rho)$, $W_2(\cdot,\mu_i)$ (see Lemma \ref{lemma:Wp-compactness}) and the fact that $\partial_i\varphi >0 $, so that $\varphi(\bm{x}) < \varphi(\bm{y})$ whenever $x_i < y_i$ for all $i=1,\dots,k$. Leveraging this fact, we deduce that
\[
\begin{split}
(\Phi^*-f^\dagger)(\pi^0) & \geq \limsup_{n \to \infty} (\Phi^*-f^\dagger)(\pi^n) \\
& \geq \limsup_{n \to \infty} (\Phi-f^\dagger)(\pi^n) \geq \liminf_{n \to \infty} (\Phi-f^\dagger)(\pi^n) \\
& \stackrel{\ref{item:subsol_optimizers_convergence_sup}}{\geq} \sup \{\Phi-f^\dagger\} = \sup (\Phi-f^\dagger)^* = \sup \{\Phi^*-f^\dagger\}.
\end{split}
\]
whence the conclusion.

\medskip

\noindent\underline{Proof of \ref{item:subsol_optimizers_convergence_valuefunction}.} As a byproduct of \ref{item:subsol_optimizers_regularization}, all inequalities in the previous display are in fact equalities, whence the first of the following relations
\[
\begin{cases}
\displaystyle{\lim_{n \to \infty}(\Phi-f^\dagger)(\pi^n) = (\Phi^*-f^\dagger)(\pi^0)} \\
\displaystyle{\limsup_{n \to \infty} \Phi(\pi^n) \leq \Phi^*(\pi^0)} \\
\displaystyle{\limsup_{n \to \infty} -f^\dagger(\pi^n) \leq -f^\dagger(\pi^0)}
\end{cases}
\]
while the second and the third ones are trivial consequences of the $W_p$-upper semicontinuity of $\Phi^*$ and $-f^\dagger$. Combining the three we obtain
\[
\limsup_{n \to \infty} \Phi(\pi^n) = \Phi^*(\pi^0) \qquad \textrm{and} \qquad \limsup_{n \to \infty} -f^\dagger(\pi^n) = -f^\dagger(\pi^0)
\]
and, up to extracting a non-relabeled subsequence (which does not affect the validity of \ref{item:subsol_optimizers_convergence_sup}),
\[
\lim_{n \to \infty} \Phi(\pi^n) = \Phi^*(\pi^0) \qquad \textrm{and} \qquad \lim_{n \to \infty} -f^\dagger(\pi^n) = -f^\dagger(\pi^0),
\]
so that \ref{item:subsol_optimizers_convergence_valuefunction} is proven too.

\medskip

\noindent\underline{Proof of \ref{item:subsol_optimizers_convergence_metric}.} It is a consequence of $f^\dagger(\mu^n) \to f^\dagger(\mu^0)$ and Lemma \ref{lemma:upgrade_semicontinuity_composition}.

\medskip

\noindent\underline{Proof of \ref{item:subsol_optimizers_convergence_time}.} Let $(\tilde\pi^n_{\cdot},\tilde u^n_{\cdot})$ be as in the statement. Because of \eqref{eq:toxicity} and item \ref{item:subsol_optimizers_convergence_sup}, we can apply Lemma \ref{lemma:raw_estimate_Phifdagger} to such a curve with $T=\tau_n$ and $\eps=T_n=\tau^2_n$. We thus obtain 
\[
\frac{1}{\tau_n} \int_0^{\tau_n} \vn{\tilde u^n_t}^2_{L^2(\tilde{\pi}^n_t)} \leq M\big(1 + \tau_n + W_2(\pi^n,\rho)\big).
\]
The first part of the claim then directly follows from item \ref{item:subsol_optimizers_convergence_metric}. To prove the fact that $\tilde\pi^n_{t_n}$ converges to $\pi^0$ w.r.t.\ $W_p$ observe that by H\"older inequality $W_p \leq W_2$ for any $p<2$, so that
\[
W_p(\tilde\pi^n_{t_n},\pi^0) \leq W_p(\pi^n,\pi^0) + W_2(\tilde\pi^n_{t_n},\pi^n).
\]
As by construction $W_p(\pi^n,\pi^0) \to 0$, it is sufficient to show that the second term on the right-hand side above vanishes. To this aim, we apply Lemma \ref{lemma:control_controlled_paths} 
\ref{item:lemma:control_controlled_paths_metriccontrol} to the curves $(\tilde\pi^n_{\cdot},\tilde u^n_{\cdot})$ with the choices $\rho=\pi^n$ to obtain
\[
W_2^2(\tilde\pi^n_{t_n},\pi^n) \leq M_n \tau_n + \int_0^{\tau_n} \vn{\tilde u^n_t}^2_{L^2(\mu^n_t)} \dd t \stackrel{\eqref{eq:reg_subs_5}}{\leq}  M_n \tau_n +M \tau_n,
\]
with $M_n = \max\{\cE(\pi^n) + c_1 W_2^2(\pi^n,\nu) + c_2,0\}$ and $c_1,c_2,\nu$ as in \eqref{eq:quadratic-lower-bound}. From the fact that $(\pi^n)_{n \geq 1}$ is a $W_2$-bounded sequence and (by \eqref{eq:reg_subs_1} and the very definition of $\overline{\cE}$) it is also bounded in energy, we deduce that $\sup_{n \geq 1} M_n<+\infty$, whence the conclusion.

Finally, the first in \eqref{eq:patience} follows from
\[
W_2(\tilde\pi^n_{t_n},\rho ) \leq W_2(\tilde\pi^n_{t_n},\pi^n) + W_2(\pi^n,\rho),
\]
and the argument outlined in the previous paragraph, which ensures that $W_2^2(\tilde\pi^n_{t_n},\pi^n) \leq O(\tau_n)$, and \ref{item:subsol_optimizers_convergence_metric}. The proof for the second in \eqref{eq:patience} is analogous.
\end{proof}

We are finally in the position to prove that $\Phi^*$ is a viscosity subsolution to \eqref{eqn:HJ_subsol_strategy}, as outlined at the beginning of the section.

\begin{proposition}\label{prop:subsolution}
Under Assumption \ref{ass: OT energy functional}, let $h \in C_b(\cP_p(\R^d))$ for some $p<2$. Then  $\Phi^*$, the upper semicontinuous relaxation w.r.t.\ the $W_p$-topology of the value function $\Phi$, defined in \eqref{def:value_fun}, is a viscosity subsolution of the Hamilton-Jacobi equation \eqref{eqn:differential_equation_tildeH1}  with $h^\dag = h$ and $A_\dag = H_\dag$ as defined in Definition \ref{definition:H}.
\end{proposition}

\begin{proof}
Note that $\Phi^*$ is bounded (since $|\Phi| \leq \|h\|_\infty$, as observed in Corollary \ref{corollary:Phi_usc_admissiblecurves}) and $W_p$-upper semicontinuous, hence $W_2$-upper semicontinuous as well. Thus, according to Definition \ref{definition:viscosity_solutions_HJ_sequences}, we are left to prove that for any $f^\dagger$, $g^\dagger$ as in Definition \ref{definition:H} there exists a sequence $(\pi_n)_{n \in \bN}$ satisfying \eqref{eqn:viscsub1} and \eqref{eqn:viscsub2}.

We claim that, given $(\tilde\pi^n_{\cdot},\tilde u^n_{\cdot})$ and $\pi^0$ as in Proposition \ref{proposition:existence_regularity_minimizer_subsol} \ref{item:subsol_optimizers_convergence_time}, the constant sequence $\pi_n = \pi^0$ is the desired one. Indeed, since by Proposition \ref{proposition:existence_regularity_minimizer_subsol} \ref{item:subsol_optimizers_regularization}
\[
\Phi^*(\pi^0)-f^{\dagger}(\pi^0) = \sup\{\Phi^*-f^\dagger\}\,,
\]
it actually suffices to show \eqref{eqn:viscsub2}, namely
\[
\Phi^*(\pi^0) - g^\dagger(\pi^0) - h(\pi^0) \leq 0.
\]
To this end, start observing that by construction and by virtue of Proposition \ref{proposition:existence_regularity_minimizer_subsol} we have 
\[
\begin{split}
\cA_{\tau_n}(\tilde\pi^n_\cdot,\tilde u^n_\cdot) + e^{-\tau_n} \Phi(\tilde\pi^n_{\tau_n})
-\Phi(\tilde\pi^n_0) \geq O(\tau^2_n) \\
(\Phi(\tilde\pi^n_0) - f^\dagger(\tilde\pi^n_0))-(\Phi(\tilde\pi^n_{\tau_n}) - f^\dagger(\tilde\pi^n_{\tau_n})) \geq O(\tau^2_n)
\end{split}
\]
where $\tau_n := n^{-1/2}$, so that multiplying the second inequality by $e^{-\tau_n}$ and adding it to the first one yields
\begin{equation}\label{eq:lower_bigO}
(e^{-\tau_n}-1)\Phi(\tilde\pi^n_0) + \cA_{\tau_n}(\tilde\pi^n_\cdot,\tilde u^n_\cdot) + e^{-\tau_n}(f^\dagger(\tilde\pi^n_{\tau_n})-f^\dagger(\tilde\pi^n_0)) \geq O(\tau^2_n).
\end{equation}
By Proposition \ref{prop:modif_EVI} and Cauchy-Schwarz inequality we have
\[
\begin{split}
f^\dagger(\tilde\pi^n_{\tau_n}) & - f^\dagger(\tilde\pi^n_0) \leq \int_0^{\tau_n} g^\dagger_{\cE}(\tilde\pi^n_t) \\
& + \int_0^{\tau_n} \|\tilde u_t\|_{L^2(\tilde\pi^n_t)} \Big(a W_2(\tilde\pi^n_t,\rho) + \nabla\varphi \bigg(\frac{1}{2} W^2_2(\tilde\pi^n_t,\bm\mu) \bigg) \cdot W_2(\tilde\pi^n_t,\bm\mu)\Big)\, \dd t,
\end{split}
\]
so that if we apply Young's inequality to the second integral on the right-hand side above, multiply by $e^{-\tau_n}$ and then add $\cA_{\tau_n}(\tilde\pi^n_\cdot,\tilde u^n_\cdot)$ to both sides, we obtain
\[
\begin{split}
\cA_{\tau_n}&(\tilde\pi^n_\cdot,\tilde u^n_\cdot) + e^{-\tau_n} \big( f^\dagger(\tilde\pi^n_{\tau_n}) - f^\dagger(\tilde\pi^n_0)\big) \\
& \leq e^{-\tau_n} \int_0^{\tau_n} g^\dagger(\tilde\pi^n_t)\,\dd t + \int_0^{\tau_n} \frac{1}{2}(e^{-\tau_n}-e^{-t}) \vn{\tilde u_t}^2_{L^2(\tilde\pi^n_t)} \dd t + \int_0^{\tau_n} e^{-t}h(\tilde \pi^n_t)\,\dd t \\
& \leq \int_0^{\tau_n}g^\dagger(\tilde \pi^n_t)\dd t +\int_0^{\tau_n}e^{-t}h(\tilde\pi^n_t)\dd t\,.
\end{split}
\]
By plugging this into \eqref{eq:lower_bigO} and dividing by $\tau_n$ we thus deduce that
\begin{equation}\label{eq:three-on-the-left}
\frac{e^{-\tau_n}-1}{\tau_n}\Phi(\tilde\pi^n_0) + \frac{1}{\tau_n}\int_0^{\tau_n} g^\dagger(\tilde\pi^n_t) \dd t + \frac{1}{\tau_n} \int_0^{\tau_n}e^{-t} h(\tilde\pi^n_t) \dd t \geq O(\tau_n).
\end{equation}
To pass to the limit as $n \to \infty$ in the three terms on the left-hand side above, note that for the first and the third it holds
\[
\begin{split}
\lim_{n \to \infty}\frac{e^{-\tau_n}-1}{\tau_n}\Phi(\tilde\pi^n_0) = -\Phi^*(\pi^0) \qquad \textrm{and} \qquad \limsup_{n \to \infty}\frac{1}{\tau_n}\int_0^{\tau_n}e^{-t}h(\tilde\pi^n_t)\dd t \leq h(\pi^0)
\end{split}
\]
thanks to Proposition \ref{proposition:existence_regularity_minimizer_subsol}\ref{item:subsol_optimizers_convergence_valuefunction} and Proposition \ref{proposition:existence_regularity_minimizer_subsol}\ref{item:subsol_optimizers_convergence_time}, respectively. More precisely, the second claim above holds true because by a change of variable
\[
\frac{1}{\tau_n}\int_0^{\tau_n} e^{-t} h(\tilde\pi^n_t)\dd t = \int_0^1 e^{-\tau_n s} h(\tilde\pi^n_{\tau_n s})\dd s
\]
and by Proposition \ref{proposition:existence_regularity_minimizer_subsol}\ref{item:subsol_optimizers_convergence_time} $W_p(\tilde\pi_{\tau_n s}^n, \pi^0) \to 0$ as $n \to \infty$, for every $s \in [0,1]$. By $W_p$-continuity of $h$, this implies $h(\tilde\pi_{\tau_n s}^n) \to h(\pi^0)$ as $n \to \infty$, for every $s \in [0,1]$, and by the fact that $h$ is bounded we can use Fatou's lemma to conclude as above.


As for the second and last term in \eqref{eq:three-on-the-left}, we argue as for the third one. Observe indeed that a change of variable yields
\[
\frac{1}{\tau_n}\int_0^{\tau_n} g^\dagger(\tilde\pi^n_t)\dd t= \int_0^1 g^\dagger(\tilde\pi^n_{\tau_n s})\dd s
\]
and on the one hand \eqref{eq:patience}, the fact that $W_p(\tilde\pi_{\tau_n s}^n,\pi^0) \to 0$ as $n \to \infty$, for every $s \in [0,1]$, and the $W_p$-lower semicontinuity of $\cE$ (see Remark \ref{rmk:lsc}) together with $\partial_i\varphi >0 $ imply
\[
\limsup_{n \to \infty} g^\dagger(\tilde\pi^n_{\tau_n s}) \leq g^\dagger(\pi^0), \qquad \forall s\in[0,1].
\]
On the other hand, \eqref{eq:uniform-control-distance}, \eqref{eq:reg_subs_5}, and Proposition \ref{proposition:existence_regularity_minimizer_subsol} \ref{item:subsol_optimizers_regularization} imply that
\[
\sup_{n \in \bN,s \in [0,1]} W_2^2(\tilde\pi^n_{\tau_n s},\rho) < +\infty, \qquad \max_{i=1,\ldots,k} \sup_{n \in \bN,s \in [0,1]} W_2^2(\tilde\pi^n_{\tau_n s},\mu_i) < +\infty
\]
and these bounds together with \eqref{eq:quadratic-lower-bound} yield
\[
\inf_{n \in \bN,s \in [0,1]} \cE(\tilde\pi^n_{\tau_n s}) > -\infty,
\]
whence
\[
\sup_{n \in \bN,s \in [0,1]}g^\dagger(\tilde\pi^n_{\tau_n s}) < +\infty.
\]
An application of Fatou's lemma provides us with
\[
\limsup_{n \to \infty} \int_0^1 g^\dagger(\tilde\pi^n_{\tau_n t}) \dd t \leq  g^\dagger(\pi_{0}),
\]
thus concluding the proof of the subsolution property. 

\end{proof}

\section{The supersolution property}\label{sec:supersolution}

In this section we work towards the proof that $\Phi_*$, the lower semicontinuous relaxation of $\Phi$ w.r.t.\ the $W_p$-topology for $p<2$, is a viscosity supersolution for 
\begin{equation} \label{eqn:HJ_supersol_strategy}
f - \lambda H_\ddagger f = h
\end{equation}
for a bounded $W_p$-continuous function $h$. Again, we will work with $\lambda = 1$.

Our strategy, to large extent, will be similar to that of the proof of the subsolution property in Section \ref{sec:subsolution}. Following however the classical proof, one replaces working with an optimally controlled curve by one that is using a control that optimizes Young's inequality as an equality for the specifically chosen test function.

As in the subsolution proof, we pick $(f^\ddagger,g^\ddagger)$ and start with almost optimizers $\pi^n_0$, namely
\begin{equation}
\Phi(\pi^n_0) - f^\ddagger(\pi^n_0) = \inf\{\Phi - f^\ddagger\} + \eps.
\end{equation}
In contrast to the subsolution proof, we do not replace $\pi^n_0$ by $\pi^n$ obtained by following optimally controlled curves, but rather extract $\pi^n$ from the gradient flow for $\cE$ started at $\pi^n_0$ to obtain control on $\cE$.

In Proposition \ref{proposition:existence_regularity_minimizer_supersol}, we show that starting from measures ever closer optimizing $\Phi - f^\ddagger$ and evolving them using the gradient flow, we can find $\pi^n$ that can be shown to have a weak limit point $\pi^0$. Finally in the proof of Proposition \ref{prop:supersolution} we show that $\pi^0$ satisfies 
\begin{equation*}
    \Phi_*(\pi^0) - g^\ddagger(\pi^0) \geq h(\pi^0).
\end{equation*}

{

\begin{proposition} \label{proposition:existence_regularity_minimizer_supersol}
Let
\[
f^{\ddagger}(\mu) := - \frac{a}{2} W_2^2(\mu,\gamma) -\varphi\left(\frac{1}{2}W^2_2(\mu,\bm\pi)\right)
\]
be an admissible test function and fix $1 \leq p < 2$. Then there exists a sequence $(\mu^n)_{n \geq 1}$ converging to some $\mu^0$ w.r.t.\ $W_p$ such that:
\begin{enumerate}[(a)]
\item \label{item:supersol_optimizers_convergence_inf} The sequence $(\mu^n)_{n \geq 1}$ is optimizing for $\inf\{\Phi - f^\ddagger\}$:
\[
\Phi(\mu^n) - f^\ddagger(\mu^n) \leq \inf \left\{\Phi - f^\ddagger \right\} + O\left(T_n\right),
\]
where $T_n = n^{-1}$. In particular,
\begin{equation} \label{eqn:supersol_optimizers_convergence_inf_nocontrol}
\lim_{n \to \infty} \{\Phi(\mu^n) - f^\ddagger(\mu^n)\} = \inf\{\Phi - f^\ddagger\}.
\end{equation}
\item \label{item:supersol_optimizers_regularization} The limit point $\mu^0$ is optimal for $\inf\{\Phi - f^\ddagger\}$ and $\inf\{\Phi_* - f^\ddagger\}$, where $\Phi_*$ denotes the lower semicontinuous relaxation of $\Phi$ w.r.t.\ the $W_p$ topology:
\[
\Phi_*(\mu^0) - f^\ddagger(\mu^0) = \inf\{\Phi_* - f^\ddagger\} = \inf\{\Phi - f^\ddagger\}.
\]
Moreover, $\mu^0 \in \cD(\cE)$.
\item \label{item:supersol_optimizers_convergence_metric} We have convergence of the metric:
\[
\lim_{n \to \infty} W_2(\mu^n,\gamma) = W_2(\mu^0,\gamma) \qquad \textrm{and} \qquad \lim_{n \to \infty} W_2(\mu^n,\pi_i) = W_2(\mu^0,\pi_i),
\]
for all $i=1,\dots,k$.
\item \label{item:supersol_optimizers_convergence_valuefunction} We have convergence of the value function:
\[
\lim_{n \to \infty} \Phi(\mu^n) = \Phi_*(\mu^0).
\]
\item \label{item:supersol_optimizers_convergence_time} Let $t_n \leq T_n$, $\psi \in \mathcal{C}_c^\infty(\R^d)$ and let $(\mu^n_\cdot,u_\cdot) \in \adm_\infty$ be the admissible curve associated to the control $u_t(x) := \nabla\psi(x)$ and initial point $\mu^n$. Then $\mu^n_{t_n}$ converges to $\mu^0$ w.r.t.\ $W_p$ as $n \to \infty$ and
\[
\lim_{n \to \infty} W_2(\mu^n_{t_n},\gamma) = W_2(\mu^0,\gamma) \qquad \textrm{and} \qquad \lim_{n \to \infty} W_2(\mu^n_{t_n},\pi_i) = W_2(\mu^0,\pi_i),
\]
for all $i=1,...,k$.
\end{enumerate}
\end{proposition}

\begin{proof}
\underline{Preliminary facts.} For $n\geq1$, let $\mu^{0,n}$ be such that 
\begin{equation}\label{eq:almost-Tn}
\Phi(\mu^{0,n}) - f^{\ddagger}(\mu^{0,n}) \leq \inf\{\Phi - f^{\ddagger}\} + T_n,
\end{equation}
where $T_n := 1/n$ as in the statement. Since $\Phi$ and $\varphi$ are lower bounded, namely $\Phi \geq -\|h\|_\infty$ and $\varphi(x_1,...,x_k) \geq \varphi(0,...,0)$ thanks to the fact that $\partial_i\varphi > 0$, we deduce that $\sup_n W_2(\mu^{0,n},\gamma) < +\infty$. By Lemma \ref{lemma:Wp-compactness} this would be enough to infer the existence of a $W_p$-convergent subsequence and a $W_p$-limit. However, no information on the energy of $\mu^{0,n}$ can be deduced from \eqref{eq:almost-Tn}. 

For this reason, we let each $\mu^{0,n}$ evolve along the gradient flow $(\mu_t^{0,n})_{t \geq 0}$ of $\cE$ and set $\mu^n := \mu^{0,n}_{T_n}$. Applying Lemma \ref{lemma:control_controlled_paths} \ref{item:lemma:control_controlled_paths_metriccontrol} with $\rho=\gamma$ to $(\mu_t^{0,n})_{t \geq 0}$ yields
\begin{equation}\label{eq:reg_super_3}
\sup_{n \in \bN} \sup_{0\leq t\leq T_n} W_2(\mu^{0,n}_t,\gamma) < +\infty.
\end{equation}
The sequence $(\mu^n)_{n \geq 1}$ is thus bounded in $(\cP_2(\R^d),W_2)$, so that by Lemma \ref{lemma:Wp-compactness} we can find $\mu^0$ and a (non-relabeled) subsequence such that $W_p(\mu^n,\mu^0) \to 0$ as $n \to \infty$. 

Moreover, $(\mu^n)_{n \geq 1}$ is also bounded in energy (which will be useful for later purposes). To prove this, note that by construction we have
\begin{equation}\label{eq:reg_super_4}
\Phi(\mu^{0,n})-\Phi(\mu^{0,n}_{T_n}) \leq f^\ddagger(\mu^{0,n}) - f^\ddagger(\mu^{0,n}_{T_n}) + T_n.
\end{equation}
As concerns the left-hand side, from Corollary \ref{corollary:Phi_usc_admissiblecurves} we know that
\begin{equation}\label{eq:reg_super_1}
\Phi(\mu^{0,n}) - \Phi(\mu^{0,n}_{T_n}) \geq -2T_n\|h\|_{\infty}.
\end{equation}
On the other hand, by \eqref{eq:reg_super_3} there exists some constant $0 < M < +\infty$ (whose numerical value changes from the first to the second line) independent of $n$ and $t \in [0,T_n]$ such that
\[
\begin{split}
-g^{\ddagger}_{\cE}(\mu_t^{0,n}) & \leq M - a\cE(\mu_t^{0,n}) - \nabla\varphi\bigg(\frac12 W_2^2(\mu_t^{0,n},\bm\pi)\bigg)\cdot \cE(\mu_t^{0,n})\bm{1} \\
& \leq M -a\cE(\mu_{T_n}^{0,n}) = M - a\cE(\mu^n),
\end{split}
\]
where in the second inequality we used the monotonicity \eqref{eq:monotonicity-entropy} of $\cE$ along the gradient flow and again \eqref{eq:reg_super_3} together with $\partial_i\varphi > 0$, \eqref{eq:quadratic-lower-bound} and the triangle inequality to bound the last term in the first line. From this control and Proposition \ref{prop:modif_EVI} we deduce that the right-hand side of \eqref{eq:reg_super_4} can be controlled from above as
\begin{equation}\label{eq:reg_super_2}
f^\ddagger(\mu^{0,n}) - f^\ddagger(\mu^{0,n}_{T_n}) \leq \int_{0}^{T_n}-g^{\ddagger}_{\cE}(\mu^{0,n}_t)\,\dd t \leq T_n(M-a\cE(\mu^{n})).
\end{equation}
Combining \eqref{eq:reg_super_1} and \eqref{eq:reg_super_2} with \eqref{eq:reg_super_4} we find
\begin{equation*}
-2T_n\|h\|_{\infty}T^{-1}_{n} \leq M - a\cE(\mu^{n}_{T_n}),
\end{equation*}
whence the desired uniform bound
\begin{equation}\label{eq:bounded-energy}
\sup_{n \in \bN}\cE(\mu^{n})<+\infty.
\end{equation}
Now, let us show that the proposed sequence enjoys properties \ref{item:supersol_optimizers_convergence_inf}-\ref{item:supersol_optimizers_convergence_time}. 

\medskip

\noindent\underline{Proof of \ref{item:supersol_optimizers_convergence_inf}.} Observe that the contraction property \eqref{eq:contraction} for $\EVI_\kappa$-gradient flows gives
\[
\begin{split}
W_2^2(\mu^n,\gamma) & \leq W_2^2(\mu^{0,n}_{T_n},\gamma_{T_n}) +  2W_2(\mu^{0,n}_{T_n},\gamma_{T_n})W_2(\gamma_{T_n},\gamma)+W_2^2(\gamma_{T_n},\gamma)\\
&\leq e^{-2\kappa T_n}W_2^2(\mu^{0,n},\gamma)+ 2e^{-\kappa T_n}W_2(\mu^{0,n},\gamma)W_2(\gamma_{T_n},\gamma)+W_2^2(\gamma_{T_n},\gamma),
\end{split}
\]
so that thanks to \eqref{eq:reg_super_3} and \eqref{eq:asymptotics} with $\mu=\nu=\gamma$ (which is possible, since $\gamma \in \cD(\cI)$ by assumption) we obtain $W_2^2(\mu^n,\gamma)-W_2^2(\mu^{0,n},\gamma) \leq O(T_n)$. Repeating the same argument for $\pi_i$, $i=1,\dots,k$, yields $W_2^2(\mu^n,\pi_i)-W_2^2(\mu^{0,n},\pi_i) \leq O(T_n)$ and combining all these bounds we get
\[
f^\ddagger(\mu^{0,n}) - f^\ddagger(\mu^n) \leq O(T_n).
\]
It is now sufficient to recall that from Corollary \ref{corollary:Phi_usc_admissiblecurves} it holds
\[
\Phi(\mu^{0,n}) - \Phi(\mu^n) \geq -2T_n\|h\|_{\infty}
\]
and thus if we further combine this bound with the previous one, we find
\[
\Phi(\mu^n)-f^{\ddagger}(\mu^n) \leq \Phi(\mu^{0,n})-f^{\ddagger}(\mu^{0,n})+ O(T_n),
\]
from which \ref{item:supersol_optimizers_convergence_inf} follows. 

\medskip

\noindent\underline{Proof of \ref{item:supersol_optimizers_regularization}.} Arguing similarly to Proposition \ref{proposition:existence_regularity_minimizer_subsol} \ref{item:subsol_optimizers_regularization}, we observe that the $W_p$-upper semicontinuity of $f^\ddagger$ implies 
\[
\begin{split}
(\Phi_* - f^\ddagger)(\mu^0) & \leq \liminf_{n \to \infty} (\Phi_* - f^\ddagger)(\mu^n) \\
& \leq \liminf_{n \to \infty} (\Phi - f^\ddagger)(\mu^n) \leq \limsup_{n \to \infty} (\Phi - f^\ddagger)(\mu^n) \\
&\stackrel{\ref{item:supersol_optimizers_convergence_inf}}{\leq} \inf\{\Phi-f^\ddagger\} = \inf (\Phi-f^\ddagger)_*= \inf\{\Phi_*-f^\ddagger\}.
\end{split}
\]
The fact that $\mu^0 \in \cD(\cE)$ is a consequence of \eqref{eq:bounded-energy}, the fact that $W_p(\mu^n,\mu^0) \to 0$ and the $W_p$-lower semicontinuity of $\cE$, discussed in Remark \ref{rmk:lsc}.

\medskip

\noindent\underline{Proof of \ref{item:supersol_optimizers_convergence_valuefunction}.} All the inequalities in the above display are in fact identities, whence the first of the following relations
\begin{equation*}
\begin{cases}
\displaystyle{\lim_{n \to \infty}(\Phi-f^\ddagger)(\mu^n) = (\Phi-f^\dagger)(\mu^0) = (\Phi_*-f^\dagger)(\mu^0)}\\
\displaystyle{\liminf_{n \to \infty} \Phi(\mu^n) \geq \Phi_*(\mu^0)} \\
\displaystyle{\liminf_{n \to \infty} -f^{\ddagger}(\mu^n) \geq -f^{\ddagger}(\mu^0)}
\end{cases}
\end{equation*}
while the remaining ones are due to the $W_p$-lower semicontinuity of $\Phi_*$ and $-f^\ddagger$. Combining the three, and up to extracting a non-relabeled subsequence, we obtain
\[
\lim_{n \to \infty} \Phi(\mu^n) = \Phi(\mu^0) \qquad \textrm{and} \qquad \lim_{n \to \infty} -f^\ddagger(\mu^n) = -f^\ddagger(\mu^0)
\]
and the first limit shows \ref{item:supersol_optimizers_convergence_valuefunction}.

\medskip

\noindent\underline{Proof of \ref{item:supersol_optimizers_convergence_metric}.} It is a consequence of $f^\ddagger(\mu^n) \to f^\ddagger(\mu^0)$ and Lemma \ref{lemma:upgrade_semicontinuity_composition}. 

\medskip

\noindent\underline{Proof of \ref{item:supersol_optimizers_convergence_time}.} 
First of all, let us stress that, given the control $u_t(x) := \nabla\psi(x)$ and the initial condition $\mu^n$, there exists a curve $\mu^n_\cdot$ such that $(\mu_\cdot,u_\cdot) \in \admerg$ and $\mu_0 = \mu^n$: it is the $\EVI$-gradient flow of $\cE(\mu) + \int\psi\,\dd\mu$. Its existence is ensured by \cite[Theorems 11.2.1 and 11.2.8]{AmGiSa08}, the standing assumptions on $\cE$ and the fact that $\nabla^2\psi \geq \kappa'{\rm Id}$ for some $\kappa' \in \R$. Let us point out that \cite[Theorem 11.2.8]{AmGiSa08} also grants that $(\mu_\cdot,u_\cdot)$ is admissible in the sense of Definition \ref{def:admissible}.

With this said, let us now prove the first part of the statement, namely the fact that $W_p(\mu^n_{t_n},\mu^0) \to 0$. We observe that by H\"older inequality $W_p \leq W_2$ for any $p<2$, so that
\[
W_p(\mu^n_{t_n},\mu^0) \leq W_2(\mu^n_{t_n},\mu^n) + W_p(\mu^n,\mu^0).
\]
As by construction $W_p(\mu^n,\mu^0) \to 0$, we focus our attention on the first term on the right-hand side. By applying Lemma \ref{lemma:control_controlled_paths} \ref{item:lemma:control_controlled_paths_metriccontrol} with $\rho = \mu^n$, it can be estimated from above as
\[
\begin{split}
W_2^2(\mu^n_{t_n},\mu^n) & \leq  \frac{2M_n}{\alpha}(1-e^{-\alpha t_n}) + \int_0^{t_n}e^{-\alpha(t_n-s)} \|\nabla\psi\|_{L^2(\mu_s^n)}^2 \dd s \\
& \leq (2M_n+\|\nabla\psi\|_\infty)\frac{1-e^{-\alpha t_n}}{\alpha}
\end{split}
\]
with $M_n = \max\{\cE(\mu^n) + c_1 W_2^2(\mu^n,\nu) + c_2,0\}$ and $c_1,c_2,\nu$ as in \eqref{eq:quadratic-lower-bound}. Since $(\mu^n)_{n \geq 1}$ is bounded in distance and in energy by the discussion carried out at the beginning of the proof, we deduce that $\sup_{n \geq 1} M_n<+\infty$ and thus $W_2(\mu_{t_n}^n,\mu^n) \to 0$ as $n \to \infty$. To prove the second part of the statement observe that
\[
W_2(\mu^n_{t_n},\gamma) \leq W_2(\mu^n_{t_n},\mu^n)+W_2(\mu^n,\gamma).
\]
By \ref{item:supersol_optimizers_convergence_metric} we know that $\lim_n W_2(\mu^n,\gamma) = W_2(\mu^0,\gamma)$ and we have already shown that $\lim_n W_2(\mu^n_{t_n},\mu^n)=0$, so that passing to the limit in the above display yields $\limsup_n W_2(\mu^n_{t_n},\gamma) \leq W_2(\mu^0,\gamma)$. By the lower semicontinuity of $W_2$ w.r.t.\ $W_p$ convergence (see Lemma \ref{lemma:Wp-compactness}) it also holds $\liminf_n W_2(\mu^n_{t_n},\gamma) \geq W_2(\mu^0,\gamma)$, whence $\lim_n W_2(\mu^n_{t_n},\gamma) = W_2(\mu^0,\gamma)$. The convergence of $W_2(\mu^n_{t_n},\pi_i)$ follows along the very same lines.
\end{proof}

We now have all the ingredients to prove that $\Phi_*$ is a viscosity supersolution to \eqref{eqn:HJ_supersol_strategy}, as anticipated at the beginning of the section.

\begin{proposition}\label{prop:supersolution}
Under Assumption \ref{ass: OT energy functional}, let $h \in C_b(\cP_p(\R^d))$ for some $p<2$. Then $\Phi_*$, the lower semicontinuous relaxation w.r.t.\ the $W_p$-topology of the value function $\Phi$, defined in \eqref{def:value_fun} is a viscosity supersolution of the Hamilton-Jacobi equation \eqref{eqn:differential_equation_tildeH2} where $h^\ddag = h$ and $A_\ddag = H_\ddag$ as defined in Definition \ref{definition:H}.
\end{proposition}

\begin{proof}
As in the proof of Proposition \ref{prop:subsolution}, we first note that $\Phi_*$ is bounded (since $|\Phi| \leq \|h\|_\infty$) and $W_2$-lower semicontinuous as well. Thus, according to Definition \ref{definition:viscosity_solutions_HJ_sequences}, we are left to prove that for any $f^\ddagger$, $g^\ddagger$ as in Definition \ref{definition:H} there exists a sequence $(\pi_n)_{n \in \bN}$ satisfying \eqref{eqn:viscsup1} and \eqref{eqn:viscsup2}.

We claim that, given $(\mu^n)$ and $\mu^0$ as in Proposition \ref{proposition:existence_regularity_minimizer_supersol}, the constant sequence $\pi_n = \mu^0$ is the desired one. Since
\[
\Phi_*(\mu^0) - f^\ddagger(\mu^0) = \inf\{\Phi_*-f^\ddagger\},
\]
it suffices to show that
\begin{equation}\label{eq:proof_supersol_7}
\Phi_*(\mu^0)-g^{\ddagger}(\mu^0)-h(\mu^0) \geq 0.
\end{equation} 
To do so, we will first derive a preliminary bound valid for any $\psi\in \cC^{\infty}_c(\R^d)$:
\begin{multline}\label{eq:corduroy_init}
\Phi_*(\mu^0) - g^\ddagger_\cE(\mu^0) - h(\mu^0)  \\
- \left( \ip{ \nabla\psi}{a\botmap{\mu^0}{\gamma} + \nabla \varphi\left(\frac12 W^2_2(\mu^0,\bm\pi) \right) \cdot \botmap{\mu^0}{\bm\pi}}_{L^2(\mu^0)} - \frac12\|\nabla\psi\|^2_{L^2(\mu^0)} \right) \geq 0.
\end{multline}
By choosing an appropriate sequence of $\psi_m$, together with an appropriate use of Cauchy-Schwarz, this will yield \eqref{eq:proof_supersol_7}.

\medskip

\noindent\underline{Proof of \eqref{eq:corduroy_init}.} We fix $\psi\in \cC^{\infty}_c(\R^d)$ and follow a similar strategy as in the proof of Proposition \ref{prop:subsolution}, but now for admissible curves $(\mu^n_\cdot,u_\cdot) \in \adm_\infty$, constructed using the suboptimal control $u_t(x) := \nabla\psi(x)$ and initial points $\mu^n$. The existence of such curves has already been discussed in the proof of Proposition \ref{proposition:existence_regularity_minimizer_supersol} \ref{item:supersol_optimizers_convergence_time}. 

From Proposition \ref{proposition:existence_regularity_minimizer_supersol} \ref{item:supersol_optimizers_convergence_inf} we directly obtain that
\begin{equation} \label{eqn:proof_supersol_almostoptimizers}
\Phi(\mu^n)-f^\ddagger(\mu^n) \leq \inf\{\Phi-f^\ddagger\} + O(T_n).
\end{equation}
Fix now a sequence $(\tau_n)_{n \geq 1}$ such that $\lim_{n}\tau_n=0$, $\lim_nT_n/\tau_n=0$. Then by \eqref{eqn:proof_supersol_almostoptimizers} we have that
\begin{equation}\label{eq:reg_super_5}
\Phi(\mu^n)-\Phi(\mu^n_{\tau_n})+f^{\ddagger}(\mu^n_{\tau_n})-f^{\ddagger}(\mu^n) \leq O(T_n).
\end{equation}
By invoking the dynamic programming principle \eqref{eq:DPP}, we obtain
\[
\Phi(\mu^n) \geq \cA_{\tau_n}(\mu^n_\cdot,\nabla\psi)+ e^{-\tau_n}\Phi(\mu^n_{\tau_n}),
\]
while the controlled EVI inequality (Proposition \ref{prop:modif_EVI}) gives
\[
f^\ddagger(\mu^n_{\tau_n})-f^\ddagger(\mu^n) \geq \int_0^{\tau_n} g^\ddagger_{\cE}(\mu^n_t) + \langle \nabla\psi, a \botmap{\mu^n_t}{\gamma} + \nabla \varphi\bigg(\frac{1}{2} W^2_2(\mu^n_t,\bm\pi)\bigg)\cdot\botmap{\mu^n_t}{\bm\pi}\rangle_{L^2(\mu^n_t)}\dd t.
\]
Plugging the last two bounds in \eqref{eq:reg_super_5} and dividing by $\tau_n$ we thus arrive at
\begin{equation}\label{eq:multiple-terms-lhs}
\begin{split}
\frac{1}{\tau_n}\cA_{\tau_n}(\mu^n_\cdot,\nabla\psi) & + \frac{e^{-\tau_n}-1}{\tau_n}\Phi(\mu^n_{\tau_n}) + \frac{1}{\tau_n}\int_0^{\tau_n} g^{\ddagger}_{\cE}(\mu^n_t) \,\dd t \\
& + \frac{1}{\tau_n}\int_0^{\tau_n} \langle \nabla\psi, a \botmap{\mu^n_t}{\gamma} + \nabla \varphi\left(\frac{1}{2}W^2_2(\mu^n_t,\bm\pi)\right)\cdot\botmap{\mu^n_t}{\bm\pi}\rangle_{L^2(\mu^n_t)}\dd t \leq O(T_n).
\end{split}
\end{equation}
We now examine each term on the left-hand side separately. 

\smallskip

\noindent\textbf{First term.} By a change of variable we note that
\[
\cA_{\tau_n}(\mu^n_\cdot,\nabla\psi) = \int_0^1 e^{-\tau_n s}\left(-\frac12 \|\nabla\psi\|^2_{L^2(\mu^n_{\tau_n s})} + h(\mu^n_{\tau_n s})\right)\dd s
\]
and remark that the modulus of the integrand function is bounded by $\frac12 \|\nabla\psi\|^2_\infty + \|h\|_\infty$. Since $h$ is $W_p$-continuous, by Proposition \ref{proposition:existence_regularity_minimizer_supersol} \ref{item:supersol_optimizers_convergence_time} and the dominated convergence theorem we thus obtain
\begin{equation}\label{eq:proof_supersol_1}
\lim_{n \to \infty} \frac{1}{\tau_n} \cA_{\tau_n}(\mu^n_\cdot,\nabla\psi) = h(\mu^0) - \frac12 \|\nabla\psi\|^2_{L^2(\mu^0)}.
\end{equation}

\smallskip

\noindent\textbf{Second term.} From the dynamic programming principle \eqref{eq:DPP} and the fact that $|\nabla\psi|$ and $h$ are bounded we have that
\[
e^{-\tau_n}\Phi(\mu^n_{\tau_n}) \leq \Phi(\mu^n) + (1-e^{-\tau_n})\Big(\frac12 \|\nabla\psi\|^2_\infty + \|h\|_\infty\Big) = \Phi(\mu^n) + O(\tau_n)
\]
and this inequality together with Proposition \ref{proposition:existence_regularity_minimizer_supersol} \ref{item:supersol_optimizers_convergence_valuefunction} yields
\begin{equation}\label{eq:proof_supersol_2}
\liminf_{n \to \infty} \frac{e^{-\tau_n}-1}{\tau_n} \Phi(\mu^n_{\tau_n}) \geq \liminf_{n \to \infty} -\Phi(\mu^n) = -\Phi_*(\mu^0).
\end{equation}

\smallskip

\noindent\textbf{Third term.} We argue as in the end of Proposition \ref{prop:subsolution}. On the one hand, by the $W_p$-lower semicontinuity of $\cE$ (see Remark \ref{rmk:lsc}) together with $\partial_i\varphi > 0$, the fact that $W_p(\mu^n_{\tau_n s}, \mu^0)$ as $n \to \infty$, for every $s \in [0,1]$, and Proposition \ref{proposition:existence_regularity_minimizer_supersol} \ref{item:supersol_optimizers_convergence_metric} we derive
\[
\liminf_{n \to \infty} g^\ddagger_\cE(\mu^n_{\tau_n s}) \geq g^\ddagger_\cE(\mu^0), \qquad \forall s \in [0,1].
\]
On the other hand, \eqref{eq:uniform-control-distance} and Proposition \ref{proposition:existence_regularity_minimizer_supersol} \ref{item:supersol_optimizers_convergence_metric} trivially imply that $W_2^2(\mu^n_{\tau_n s},\gamma)$ and $W_2^2(\mu^n_{\tau_n s},\pi_i)$ are bounded uniformly in $n \in \bN$, $s \in [0,1]$, and $i=1,\dots,k$. As a consequence, $\cE(\mu^n_{\tau_n s})$ is bounded from below uniformly in $n \in \bN$ and $s \in [0,1]$, whence
\[
\inf_{n \in \bN,s \in [0,1]} g^\ddagger_\cE(\mu^n_{\tau_n s}) > -\infty.
\]
A change of variable and Fatou's lemma then give
\begin{equation}\label{eq:proof_supersol_3}
\liminf_{n \to \infty}\frac{1}{\tau_n} \int_0^{\tau_n} g^\ddagger_{\cE}(\mu^n_t) \,\dd t = \liminf_{n \to \infty} \int_0^1 g^\ddagger_\cE(\mu^n_{\tau_n s})\,\dd s \geq g^\ddagger_{\cE}(\mu^0).
\end{equation}

\smallskip

\noindent\textbf{Fourth term.} Recall that for any $\eta \in \{\gamma, \pi_1,\dots,\pi_k\}$, $\otmap{\eta}{\mu^n_t}$ is the optimal transport map pushing $\eta$ onto $\mu^n_t$ and
\begin{equation*}
    \botmap{\mu^n_t}{\eta} \circ \otmap{\eta}{\mu^n_t} =  \left( \otmap{\mu^n_t}{\eta} -\bm{id}\right) \circ \otmap{\eta}{\mu^n_t} = -\botmap{\eta}{\mu^n_t}.
\end{equation*}
After a change of variable, the last term in \eqref{eq:multiple-terms-lhs} can thus be equivalently rewritten as
\[
\begin{split}
\frac{1}{\tau_n}\int_0^{\tau_n} & \langle \nabla\psi, a \botmap{\mu^n_t}{\gamma} + \nabla \varphi\left(\frac{1}{2}W^2_2(\mu^n_t,\bm\pi)\right)\cdot\botmap{\mu^n_t}{\bm\pi}\rangle_{L^2(\mu^n_t)}\dd t \\
& = -\int_0^1 \ip{\nabla\psi \circ \otmap{\gamma}{\mu^n_{\tau_n s}}}{a\botmap{\gamma}{\mu^n_{\tau_n s}}}_{L^2(\gamma)} \dd s \\
& \qquad - \sum_{i=1}^k \int_0^1 \ip{\nabla\psi \circ \otmap{\pi_i}{\mu^n_{\tau_n s}}}{\partial_i \varphi\left(\frac12 W^2_2(\mu^n_{\tau_n s},\bm\pi) \right) \botmap{\pi_i}{\mu^n_{\tau_n s}}}_{L^2(\pi_i)} \dd s.
\end{split}
\]
To obtain convergence as $n \to \infty$ to the desired term (i.e.\ the second line in \eqref{eq:corduroy_init}) we apply the dominated convergence theorem. We first claim that
\begin{equation}\label{eq:pointwise-conv-otmap}
\langle \nabla\psi \circ \otmap{\eta}{\mu^n_{\tau_n s}},\botmap{\eta}{\mu^n_{\tau_n s}}\rangle_{L^2(\eta)} \to \langle \nabla\psi \circ \otmap{\eta}{\mu^0},\botmap{\eta}{\mu^0}\rangle_{L^2(\eta)}
\end{equation}
for all $s \in [0,1]$ and $\eta \in \{\gamma,\pi_1,\dots,\pi_k\}$. Indeed, we note that
\[
\begin{split}
|\langle \nabla\psi \,\circ\, & \otmap{\eta}{\mu^n_{\tau_n s}}, \botmap{\eta}{\mu^n_{\tau_n s}}\rangle_{L^2(\eta)} - \langle \nabla\psi \circ \otmap{\eta}{\mu^0},\botmap{\eta}{\mu^0}\rangle_{L^2(\eta)}| \\
& \leq |\langle \nabla\psi \circ \otmap{\eta}{\mu^n_{\tau_n s}},\botmap{\eta}{\mu^n_{\tau_n s}}\rangle_{L^2(\eta)} - \langle \nabla\psi \circ \otmap{\eta}{\mu^n_{\tau_n s}},\botmap{\eta}{\mu^0}\rangle_{L^2(\eta)}| \\
& \quad + |\langle \nabla\psi \circ \otmap{\eta}{\mu^n_{\tau_n s}},\botmap{\eta}{\mu^0}\rangle_{L^2(\eta)} - \langle \nabla\psi \circ \otmap{\eta}{\mu^0},\botmap{\eta}{\mu^0}\rangle_{L^2(\eta)}| \\
& \leq \|\nabla\psi\|_\infty \|\otmap{\eta}{\mu^n_{\tau_n s}} - \otmap{\eta}{\mu^0}\|^2_{L^2(\eta)} + \|\nabla\psi \circ \otmap{\eta}{\mu^n_{\tau_n s}} - \nabla\psi \circ \otmap{\eta}{\mu^0}\|_{L^2(\eta)}\|\botmap{\eta}{\mu^n_0}\|_{L^2(\eta)}.
\end{split}
\]
Then, by Proposition \ref{proposition:existence_regularity_minimizer_supersol} \ref{item:supersol_optimizers_convergence_time} and Lemma \ref{lem:continuity_map} the first term on the last line vanishes as $n \to \infty$ and, up to extracting a subsequence, $\otmap{\eta}{\mu^n_{\tau_n s}} \to \otmap{\eta}{\mu^0}$ $\eta$-a.e. Since $\nabla\psi$ is continuous and bounded, by the dominated convergence theorem also the second term goes to 0, whence \eqref{eq:pointwise-conv-otmap}.

Moreover, by Proposition \ref{proposition:existence_regularity_minimizer_supersol} \ref{item:supersol_optimizers_convergence_time} $W_2(\mu^n_{\tau_n s},\eta)$ is bounded uniformly in $n \in \bN$ and $s \in [0,1]$, so that if we set
\[
R := \max_{\eta \in \{\gamma,\pi_1,\dots,\pi_k\}}\sup_{s \in [0,1]}\sup_{n \in \bN} W_2(\mu^n_{\tau_n s},\eta) < +\infty,
\]
we deduce that
\[
|\langle \nabla\psi \circ \otmap{\gamma}{\mu^n_{\tau_n s}},a\botmap{\gamma}{\mu^n_{\tau_n s}}\rangle_{L^2(\gamma)}| \leq a\|\nabla\psi\|_\infty W_2(\mu^n_{\tau_n s},\gamma) \leq a\|\nabla\psi\|_\infty R
\]
and
\[
\begin{split}
\left|\ip{\nabla\psi \circ \otmap{\pi_i}{\mu^n_{\tau_n s}}}{\partial_i \varphi\left(\frac12 W^2_2(\mu^n_{\tau_n s},\bm\pi) \right) \botmap{\pi_i}{\mu^n_{\tau_n s}}}_{L^2(\pi_i)}\right| & \leq \partial_i\varphi(R^2/2)\|\nabla\psi\|_\infty  W_2(\mu^n_{\tau_n s},\pi_i) \\
& \leq \partial_i\varphi(R^2/2)\|\nabla\psi\|_\infty R.
\end{split}
\]
We can thus apply the dominated convergence theorem and get
\begin{equation}\label{eq:proof_supersol_4}
\begin{split}
\lim_{n \to \infty} \frac{1}{\tau_n} & \int_0^{\tau_n} \langle \nabla\psi, a \botmap{\mu^n_t}{\gamma} + \nabla \varphi\left(\frac{1}{2}W^2_2(\mu^n_t,\bm\pi)\right)\cdot\botmap{\mu^n_t}{\bm\pi}\rangle_{L^2(\mu^n_t)}\dd t \\
& = -a \langle \nabla\psi \circ \otmap{\gamma}{\mu^0},\botmap{\gamma}{\mu^0}\rangle_{L^2(\gamma)} - \sum_{i=1}^k \partial_i\varphi\left(\frac12 W_2^2(\mu^0,\bm\pi)\right) \langle \nabla\psi \circ \otmap{\pi_i}{\mu^0},\botmap{\pi_i}{\mu^0}\rangle_{L^2(\pi_i)} \\
& = \langle \nabla\psi,a \botmap{\mu^0}{\gamma} + \nabla\varphi\left(\frac12 W_2^2(\mu^0,\bm\pi)\right) \cdot \botmap{\mu^0}{\bm\pi}\rangle_{L^2(\mu^0)},
\end{split}
\end{equation}
where in the last identity we used again the definition of transport map to come back to an integral w.r.t.\ $\mu^0$. To justify this change of variable, however, we first need to show that $\botmap{\mu^0}{\gamma}$ and $\botmap{\mu^0}{\bm\pi}$ actually exist. But this is ensured by Proposition \ref{proposition:existence_regularity_minimizer_supersol} \ref{item:supersol_optimizers_regularization} (which implies $\mu^0 \ll \cL^d$) and \cite[Theorem 6.2.4]{AmGiSa08}.

If we now gather \eqref{eq:proof_supersol_1}, \eqref{eq:proof_supersol_2}, \eqref{eq:proof_supersol_3}, and \eqref{eq:proof_supersol_4} we obtain
\begin{equation}\label{eq:corduroy}
\begin{split}
h(\mu^0) - \frac12\|\nabla\psi\|^2_{L^2(\mu^0)} & - \Phi_*(\mu^0) + g^\ddagger_\cE(\mu^0) +  \\
& + \ip{ \nabla\psi}{a\botmap{\mu^0}{\gamma} + \nabla \varphi\left(\frac12 W^2_2(\mu^0,\bm\pi) \right) \cdot \botmap{\mu^0}{\bm\pi}}_{L^2(\mu^0)} \leq 0,
\end{split}
\end{equation}
which, up to rearrangement of terms, is \eqref{eq:corduroy_init}. 

\medskip

\noindent\underline{Proof of \eqref{eq:proof_supersol_7}.} Since $\mu^0 \in \cD(\cE)$, we can invoke \cite[Proposition 8.5.2]{AmGiSa08}, which ensures that
\[
\botmap{\mu^0}{\gamma},\,\botmap{\mu^0}{\pi_i} \in \overline{\{\nabla\psi \,:\, \psi \in \cC^\infty_c(\R^d)\}}^{L^2(\mu^0)}, \qquad i=1,...,k.
\]
We can thus choose $\psi_m \in \mathcal{C}_c^\infty(\R^d)$ such that 
\[
\nabla\psi_m \to a\botmap{\mu^0}{\gamma} + \nabla \varphi\left(\frac12 W^2_2(\mu^0,\bm\pi) \right) \cdot \botmap{\mu^0}{\bm\pi} \qquad \textrm{in } L^2(\mu^0)
\]
and note once again that \eqref{eq:corduroy} holds true for all $\psi \in \mathcal{C}_c^\infty(\R^d)$. Therefore, if we consider \eqref{eq:corduroy} with $\psi=\psi_m$ and take the limit as $m \to \infty$, we obtain
\begin{equation}\label{eq:proof_supersol_5}
h(\mu^0) - \Phi_*(\mu^0) + g^\ddagger_\cE(\mu^0) + \frac12 \left\|a\botmap{\mu^0}{\gamma} + \nabla \varphi\left(\frac12 W^2_2(\mu^0,\bm\pi) \right) \cdot \botmap{\mu^0}{\bm\pi}\right\|^2_{L^2(\mu_0)} \leq 0.
\end{equation}
To finish the proof, Cauchy-Schwarz inequality and the fact that $\|\botmap{\mu^0}{\eta}\|_{L^2(\mu^0)} = W_2(\mu^0,\eta)$ for $\eta \in \{\gamma,\pi_1,\dots,\pi_k\}$ provide us with
\[
\begin{split}
\frac12 \bigg\|a\botmap{\mu^0}{\gamma} & + \nabla \varphi\left(\frac12 W^2_2(\mu^0,\bm\pi) \right) \cdot \botmap{\mu^0}{\bm\pi}\bigg\|_{L^2(\mu^0)}^2 \\
& = \frac{a^2}{2}W^2_2(\mu_0,\gamma) + a \langle \botmap{\mu^0}{\gamma}, \nabla \varphi\left(\frac12 W^2_2(\mu^0,\bm\pi) \right) \cdot \botmap{\mu^0}{\bm\pi} \rangle_{L^2(\mu^0)}\\
& \quad + \frac12\left|\nabla \varphi\left(\frac12 W^2_2(\mu^0,\bm\pi) \right) \cdot \botmap{\mu^0}{\bm\pi}\right|^2_{L^2(\mu_0)}\\
& \geq \frac{a^2}{2} W^2_2(\mu_0,\gamma) -a W_2(\mu^0,\gamma)\left( \nabla \varphi\left(\frac12 W^2_2(\mu^0,\bm\pi) \right) \cdot W_2(\mu^0,\bm\pi)\right)\\
& \quad -\frac{1}{2}\left( \nabla\varphi\Big(\frac{1}{2}W_2^2(\mu^0,\bm\pi)\Big)\cdot W_2(\mu^0,\bm\pi)\right)^2\\
& = g^{\ddagger}_{W_2}(\mu^0).
\end{split}
\]
Plugging this bound into \eqref{eq:proof_supersol_5} gives the desired conclusion \eqref{eq:proof_supersol_7}.
\end{proof}



\appendix

\section{Appendix} \label{section:appendix}

\begin{lemma}\label{lemma:upgrade_semicontinuity_composition}
Let
\[
f^\dagger(\pi) = \frac{a}{2} W_2^2(\pi,\rho) + \varphi\left(\frac{1}{2} W_2^2(\pi,\bm\mu) \right)
\]
with $a > 0$, $\varphi \in \cT$ as in \eqref{eqn:defT} and $\rho,\mu_1,\dots,\mu_k \in \cP_2(\bR^d)$. Suppose that $\pi^n \rightarrow \pi^0$ weakly and $f^\dagger(\pi^n) \rightarrow f^\dagger(\pi^0)$. Then we have
\[
W_2(\pi^n,\rho) \to W_2(\pi^0,\rho), \qquad \textrm{and} \qquad W_2(\pi^n,\mu_i) \to W_2(\pi^0,\mu_i)
\]
for all $i=1,\ldots,k$.
\end{lemma}

\begin{proof}
By lower semicontinuity w.r.t.\ weak convergence we know that 
\begin{equation}\label{eq:lsc_metric}
W_2(\pi^0,\rho) \leq \liminf_{n \to \infty} W_2(\pi^n,\rho) \qquad \textrm{and} \qquad W_2(\pi^0,\mu_i) \leq \liminf_{n \to \infty} W_2(\pi^n,\mu_i)
\end{equation}
for all $i=1,\ldots,k$ and, together with the fact that $\varphi$ is continuous and increasing in each coordinate, this implies that we can find a subsequence $(n_m)_{m \in \bN}$ along which all the above liminf's are in fact limits and
\begin{equation}\label{eq:lsc_varphi}
\varphi\bigg(\frac12 W^2_2(\pi^0,\bm\mu)\bigg) \leq \varphi\bigg(\frac12 \liminf_{m \to \infty} W^2_2(\pi^{n_m},\bm\mu)\bigg) = \lim_{m \to \infty} \varphi\bigg(\frac12 W^2_2(\pi^{n_m},\bm\mu)\bigg),
\end{equation}
where
\[
\liminf_{m \to \infty} W^2_2(\pi^{n_m},\bm\mu) = \Big(\liminf_{m \to \infty} W_2^2(\pi^{n_m},\mu_1), \ldots, \liminf_{m \to \infty} W_2^2(\pi^{n_m},\mu_k)\Big).
\]
If we assume by contradiction that $W_2(\pi^0,\mu_i) < \lim_m W_2(\pi^{n_m},\mu_i)$ for some $i$, then by the fact that $\partial_i\varphi > 0$, i.e.\ $\varphi$ is coordinatewise strictly increasing, the inequality in \eqref{eq:lsc_varphi} must be strict. Recalling that $W_2(\pi^0,\rho) \leq \lim_m W_2(\pi^{n_m},\rho)$ by \eqref{eq:lsc_metric}, it is now sufficient to multiply it by $a/2$ and sum it with \eqref{eq:lsc_varphi}, thus obtaining
\[
\begin{split}
\frac{a}{2} W_2^2(\pi^0,\rho) + \varphi\bigg(\frac12 W^2_2(\pi^0,\bm\mu)\bigg) < \lim_{m \to \infty} \frac{a}{2} W_2^2(\pi^{n_m},\rho) + \varphi\bigg(\frac12 W^2_2(\pi^{n_m},\bm\mu)\bigg),
\end{split}
\]
namely
\[
f^\dagger(\pi^0) < \lim_{m \to \infty}f^\dagger(\pi^{n_m})
\]
and this is a contradiction to $f^\dagger(\pi^n) \to f^\dagger(\pi^0)$. Hence we can conclude that $\limsup_n W_2(\pi^n,\mu_i) = \lim_m W_2(\pi^{n_m},\mu_i) \leq W_2(\pi^0,\mu_i)$ and thus $W_2(\pi^n,\mu_i) \to W_2(\pi^0,\mu_i)$ for all $i=1,\ldots,k$. The proof of $W_2(\pi^n,\rho) \to W_2(\pi^0,\rho)$ follows an analogous (but simpler) argument.
\end{proof}

}

\printbibliography

\end{document}